\documentclass[12pt,a4paper,oneside,reqno]{amsart}

\usepackage{amsthm}
\usepackage{amsmath}
\usepackage{amssymb}
\usepackage{amsfonts}
\usepackage{latexsym}
\usepackage{amscd}
\usepackage[utf8]{inputenc}
\usepackage{typearea}
\usepackage{yfonts}
\usepackage{textcomp}
\usepackage{mathrsfs}
\usepackage{hyperref}
\usepackage[draft]{fixme} 
\usepackage{pdfsync}
\usepackage[active]{srcltx}
\usepackage{xypic}
\usepackage{tikz}
\usetikzlibrary{arrows.meta,cd,positioning}
\usetikzlibrary{cd}
\usepackage{mathrsfs}
\usepackage{verbatim}
\usepackage{tensor}
\usetikzlibrary{calc}

\newcommand{\ZZ}{\mathbb{Z}}

\newcommand{\GrpH}{\mathcal{H}}
\newcommand{\GrpG}{\mathcal{G}}

\newcommand{\uloopr}[1]{\ar@'{@+{[0,0]+(-4,5)}@+{[0,0]+(0,10)}@+{[0,0] +(4,5)}}^{#1}}
\newcommand{\uloopd}[1]{\ar@'{@+{[0,0]+(5,4)}@+{[0,0]+(10,0)}@+{[0,0]+ (5,-4)}}^{#1}}
\newcommand{\dloopr}[1]{\ar@'{@+{[0,0]+(-4,-5)}@+{[0,0]+(0,-10)}@+{[0, 0]+(4,-5)}}_{#1}}
\newcommand{\dloopd}[1]{\ar@'{@+{[0,0]+(-5,4)}@+{[0,0]+(-10,0)}@+{[0,0 ]+(-5,-4)}}_{#1}}

\newtheorem{lemma}{Lemma}[section]
\newtheorem{corollary}[lemma]{Corollary}
\newtheorem{theorem}[lemma]{Theorem}
\newtheorem{proposition}[lemma]{Proposition}

\theoremstyle{definition}
\newtheorem{definition}[lemma]{Definition}

\newtheorem{example}[lemma]{Example}

\newtheorem{notation}[lemma]{Notation}
\newtheorem{remark}[lemma]{Remark}

\theoremstyle{theorem}
\newtheorem{introtheorem}{Theorem}

\numberwithin{equation}{section}

\title[Triviality of groupoid extensions]{Global and local triviality for extensions of ample groupoids: an algebraic approach}

\begin{document}

	\author{Nathan Brownlowe}
	\address{School of Mathematics and Statistics\\
		The University of Sydney\\
		NSW 2006\\
		Australia} \email{nathan.brownlowe@sydney.edu.au}\urladdr{https://www.maths.usyd.edu.au/u/nathanb/}
	
	\author{Lisa Orloff Clark}
	\address{School of Mathematics and Statistics\\
		Victoria University of Wellington\\
		PO Box 600\\
		Wellington 6140\\
		New Zealand} \email{lisa.orloffclark@vuw.edu.nz}\urladdr{https://people.wgtn.ac.nz/lisa.orloffclark}
	
	\author{Enrique Pardo}
	\address{Departamento de Matem\'aticas, Facultad de Ciencias\\ Universidad de C\'adiz, Campus de
		Puerto Real\\ 11510 Puerto Real (C\'adiz)\\ Spain}
	\email{enrique.pardo@uca.es}\urladdr{https://sites.google.com/gm.uca.es/webofenriquepardo/inicio}
	
	\author{Aidan Sims}
	\address{School of Mathematics and Statistics\\
		University of New South Wales\\
		NSW 2052\\
		Australia} \email{aidan.sims@unsw.edu.au}\urladdr{https://www.aidansims.com}
	
	
	\thanks{The first-named author was partially supported by Australian Research Council Discovery Project DP250100297. The second-named author was partially supported by the Marsden Fund of the Royal Society of New Zealand (grant number 21-VUW-156). The third-named author was partially supported by the Spanish State Research Agency (grant numbers PID2020-113047GB-I00 and PID2023-147110NB-I00), by PAIDI grants FQM-298 and ProyExcel 00780 ``Operator Theory: an interdisciplinary approach'' (2021), of the Junta de Andaluc\'ia. The fourth-named author was partially supported by Australian Research Council Discovery Projects DP220101631 and DP250100297.}
	
	
	

	\begin{abstract}
		Using the duality between the categories of inverse semigroups and ample groupoids \cite{S2}, we prove the triviality (global and local) of the central term of the extension in terms of twisted crossed products.
	\end{abstract}
	
	\maketitle
	
	\section{Introduction}
	
The study of algebras associated with groups and their
generalizations---(inverse) semigroups, groupoids, and so forth---has been a
key tool for understanding their structure. A classical example is the class
of group algebras that connect the representation theory of the group to that
of the corresponding group algebra \cite{Pass1}.

In the case of group algebras, the extensions of groups are connected to a
natural algebra construction: (twisted) crossed products. Indeed, given a
group extension
$$N\hookrightarrow H \twoheadrightarrow G, $$
we can represent the group algebra $K[H]$ as a twisted crossed product
$K[N]\ast_{\sigma,\tau} G$: the action $\sigma:G\rightarrow \text{Aut}(K[N])$
is induced by the canonical action of $G$ by conjugation on $N$; and the
2-cocycle $\tau:G\times G\rightarrow \mathcal{U}(K[N])$ appears as the
obstruction to a section of the quotient map from $G$ to $K$ being a
homomorphism (see, for example, \cite{Pass2}). In fact, $H \cong N
*_{\sigma,\tau} G$, and so
$$K[N]\ast_{\sigma,\tau} G \cong K[H] \cong K[N\ast_{\sigma,\tau} G].$$
That is, in the case of groups, the central term of a group extension can be
realised as a twisted crossed product of the extremes of this extension, and
this decomposition passes in a natural way to group rings.

We are interested in the analogous circle of ideas for groupoids. In the
setting of $C^*$-algebras, groupoids have played a major role since Renault's
seminal work \cite{Renault}. Here, natural generalizations of the above group
crossed product algebras arise from the results of Renault \cite{Renault} and
Kumjian \cite{Kumjian}, and the twisted groupoid algebras are abstractly
characterised as Cartan pairs of $C^*$-algebras. There are two constructions
associated with this phenomenon: on the one hand, an algebra associated with
twisting the convolution product of the groupoid algebra $C^*(\mathcal{G})$
using a continuous 2-cocycle $\sigma:\mathcal{G}{ }_s\times_r
\mathcal{G}\rightarrow \mathbb{T}$ taking values in the group $\mathbb{T}$ of
unimodular complex numbers; and on the other hand a subalgebra of equivariant
elements of the groupoid $C^*$-algebra $C^*(\mathcal{G}\times_\sigma
\mathbb{T})$, where the groupoid arises as the central term of a groupoid
extension
$$\mathcal{G}^{(0)}\times \mathbb{T}\hookrightarrow \mathcal{G}\times_\sigma \mathbb{T} \twoheadrightarrow \mathcal{G}.$$
Kumjian extended this construction to extensions
$$\mathcal{G}^{(0)}\times \mathbb{T}\hookrightarrow \Sigma \twoheadrightarrow \mathcal{G}$$
in which the central term $\Sigma$ need not be a twisted product of
$\mathcal{G}$ by $\mathbb{T}$. When $\mathcal{G}$ is Hausdorff and totally
disconnected, the two definitions are equivalent.

The interest of this question passed to  the purely algebraic setting through
the algebraic siblings of groupoid $C^*$-algebras for ample groupoids, known
as Steinberg algebras \cite{S1}. The second and then the second and fourth
authors (and many coauthors), extended the notion of a twisted groupoid
algebra to the case of Steinberg algebras  in \cite{A6, A9}, following the
$C^*$-algebraic scheme of construction. Here $\mathbb{T}$ can be replaced by
a group of units of a general commutative ring, and is always regarded as
carrying the discrete topology. The upshot is that the algebras obtained from
the construction are far more general than in the $C^*$-algebraic setting.

In \cite{S2}, Steinberg formulated a new approach to twisted Steinberg
algebras, by showing that the Lawson--Lenz duality between ample groupoids
and Boolean inverse semigroups extends to a category equivalence preserving
extensions on both sides. He used this approach to define a generalization of
twisted Steinberg algebras by suitable extensions of ample groupoids in which
the normal kernel turns out to be an abelian subgroupoid, so having a twisted
product central term. But to extend his construction to its full
generality---as in the group algebra case---one must first understand how to
characterise triviality of the central term of an \emph{arbitrary} extension
of ample groupoids without restrictions, and then to find a correct algebraic
model for \emph{twisted crossed products} of Steinberg algebras. In this
paper, we address the first question. \vspace{.2truecm}
	
	Following \cite{S2}, by an extension of ample groupoids we mean an exact sequence
	\[
	\xymatrix{\GrpH\ar[r]^{\kappa}  & \Sigma \ar[r]^{\rho} & \GrpG}
	\]
	of open, iso-unital functors with $\kappa$ injective, $\rho$ surjective, and
	$\rho^{-1}(\GrpG^{(0)})=\kappa(\GrpH)$. Observe that since $\rho \circ \kappa$ is a homeomorphism of unit-spaces, $r \circ \kappa = s\circ \kappa$ and so $\mathcal{H}$ is a group bundle.
    Our aim in this paper is to study the triviality properties of such an extension, in the sense of $\Sigma$ being coordinatised by the kernel subgroupoid $\GrpH$ and the quotient groupoid $\GrpG$.
	
	In our first main result we show that when there exists a continuous global section $\gamma:\GrpG\to \Sigma$ of $\rho$, we can describe $\Sigma$ globally as a twisted crossed-product groupoid.
	
	\begin{introtheorem}\label{thm:mainthm1}
		Suppose
		\[
		\xymatrix{\GrpH\ar[r]^{\kappa}  & \Sigma \ar[r]^{\rho} & \GrpG \ar@/^8pt/ [l] ^{\gamma}},
		\]
		is an extension of ample groupoids, with $\gamma$ a continuous global section of $\rho$. Then there exists a triple $\widehat{\Lambda}:=(\widehat{\alpha},\widehat{\lambda},\widehat{f})$ consisting of a homeomorphism $\widehat{\alpha}\colon \GrpG^{(0)}\to \GrpH^{(0)}$; a continuous $\widehat{f}$-twisted action $\widehat{\lambda}\colon \GrpG\to \operatorname{End}(\GrpH)$ where each $\widehat{\lambda}_g$ is a bijection $\GrpH_{\widehat{\alpha}(d_\GrpG(g))}\to \GrpH_{\widehat{\alpha}(r_\GrpG(g))}$; and a continuous $2$-cocycle $\widehat{f}\colon \GrpG^{(2)}\to \GrpH$ for $\widehat{\lambda}$, with the following property: the set
		\[
		\GrpH\ast_{\widehat{\Lambda}}\GrpG:=\{(h,g)\in \GrpH\times \GrpG : d_\GrpH(h)=\widehat{\alpha}(r_\GrpG(g))\}
		\]
		is a groupoid with $(\GrpH\ast_{\widehat{\Lambda}}\GrpG)^{(2)} = \{((h, g), (h',g')) : s(h) = \widehat{\lambda}(r(h'))\text{ and } s(g) = r(g')\}$ and multiplication given by
		\begin{equation}\label{eq:product formula}
		(h_1,g_1)(h_2,g_2)
		=
		\bigl(h_1\,\widehat{\lambda}_{g_1}(h_2)\,\widehat{f}(g_1,g_2),\,g_1g_2\bigr)\text{ for }(g_1,g_2)\in \GrpG^{(2)};
		\end{equation}
        moreover, under the relative topology inherited from $\GrpH \times \GrpG$ it is a topological groupoid and is isomorphic to $\Sigma$.
	\end{introtheorem}
	
	In general, however, such a global section $\gamma$ need not exist, as we illustrate in Section~\ref{Sect:Examples}. In this case, we replace the global crossed-product model by a local coordinatisation (see Remark~\ref{rmk:localthmforgroupoids}): the groupoid $\GrpH\ast_{\widehat{\Lambda}}\GrpG$ is replaced by a space $\mathcal{X}$ that can be identified set-theoretically with $\GrpH\ast_{\widehat{\Lambda}}\GrpG$, and that is covered by local coordinate spaces $\mathcal{X}_U$ that can be identified set-theoretically with compact open bisections $U$ of $\GrpH\ast_{\widehat{\Lambda}}\GrpG$ and on pairs of which a local product formula akin to~\eqref{eq:product formula} describes multiplication in $\Sigma$. So Theorem~\ref{thm:mainthmnoorderpreserving} and Remark~\ref{rmk:localthmforgroupoids} can be viewed as a local version of Theorem~\ref{thm:mainthm1}.

\section{Preliminaries}\label{sec:prelim}
	
After some basic preliminaries, we begin by outlining Steinberg's work
\cite{S2} on the duality between the categories of ample groupoids and
Boolean inverse semigroups. We then present the key definitions and results
from Paterson's universal groupoid construction, and the theory of inverse
semigroup extensions.

On one side of our duality, we consider groupoid extensions.  A groupoid is a
small category in which every morphism has an inverse; see, for example,
\cite{S1} for more details. Throughout this paper, the groupoids we consider
are topological groupoids with Hausdorff unit spaces. We also assume that
they are ample, meaning that they have a basis of compact open subsets on
which the source and range maps restrict to homeomorphisms onto open subsets
of the unit space. We refer to such subsets as compact open bisections.

On the other side of our duality, we consider inverse semigroup extensions.
See \cite{Law} for the background on inverse semigroups. Given an element $s$
in an inverse semigroup $S$, we denote its inverse by $s^*$ and we denote the
set of all idempotents in $S$ by $E(S)$.    Recall that an inverse semigroup
$S$ comes with a natural partial order where $s\leq t$ if and only if
$ss^*t=s$. We call $S$ a Clifford inverse semigroup if idempotents are
central, that is, for any idempotent $e \in S$, and any $s \in S$, we have
$se=es$.
	
	\subsection{Categorical duality}\label{Subsect:Categoricalduality}
	
	In \cite{S2}, Steinberg studied the problem of defining and reconstructing twisted Steinberg algebras \cite{A6, A9} through an inverse semigroup approach.
	At the level of algebras this makes sense because of \cite[Theorem 6.3]{S1}: If $R$ is a commutative unital ring, $S$ is an inverse semigroup, and $\GrpG(S)$ is the universal Paterson groupoid (see below), then there exists a canonical $R$-algebra isomorphism $R[S]\cong A_R(\GrpG(S))$; moreover, $R[-]$ defines a covariant functor from the category of inverse semigroups and $\ast$-semigroup homomorphisms, to the category of $R$-algebras and (unital) homomorphisms.
	
	The fundamental result of Steinberg is that Lawson--Lenz duality between Boolean inverse semigroups and ample groupoids preserves extensions in both categories. Concretely, Steinberg shows the following.
	\begin{enumerate}
		\item If $\GrpG$ is an ample groupoid, then the set $\Gamma_c(\GrpG)$ of compact open bisections in $\GrpG$ is a Boolean inverse semigroup.
		\item If $S$ is a Boolean inverse semigroup, then $\GrpG(S)$ is an ample groupoid.
		\item The maps $\GrpG(-)$ and $\Gamma_c(-)$ are mutually inverse \cite[Theorem 3.3]{S2}, and induce categorical equivalences between the categories of ample groupoids with iso-unital functors (maps that restrict to homeomorphisms of unit spaces), and Boolean inverse semigroups with  idempotent bijective homomorphisms (maps that restrict to isomorphisms of the semilattice of idempotents) \cite[Theorem 3.4]{S2}. Moreover, these functors preserve injective and surjective maps.
	\end{enumerate}
	These results mean that the functors $\GrpG(-)$ and $\Gamma_c(-)$ constitute a duality between, on the one hand extensions of ample groupoids
	$$
	\xymatrix{\GrpH\ar[r]^{\kappa}  & \Sigma \ar[r]^{\rho} & \GrpG},
	$$
	with $\kappa, \rho$ open iso-unital functors, $\kappa$ injective, $\rho$ surjective, and $\rho^{-1}(\GrpG^{(0)})=\kappa(\GrpH)$, and on the other hand, extensions of Boolean inverse semigroups
	$$
	\xymatrix{N\ar[r]^{\iota}  & T \ar[r]^{\pi} & S}
	$$
	with $\iota, \pi$ idempotent-bijective homomorphisms, $\iota$ injective, $\pi$ surjective, and $\iota (N)=\pi^{-1}(E(S))$.
	
	A key result for our work is \cite[Lemma 3.7]{S2}, which says that $\GrpG(-)$ and $\Gamma_c(-)$ identify continuous global sections of surjective open iso-unital functors between ample groupoids with order-preserving sections of idempotent-bijective homomorphisms between Boolean inverse semigroups.
	
	According to \cite[Chapter 5]{Law}, since the maps are idempotent-separating, $N$ is a Clifford semigroup (so that $N=Z(E(T))$) and it is a normal kernel for $\pi$ (i.e. $N$ is stable under conjugacy in $T$, and ``generates'' $\ker \pi$ as a congruence).
	
	Steinberg only added some extra restrictions ($N$ being abelian, or $\GrpH=\GrpG^{(0)} \times T$ for $T$ an abelian group) to apply his results to the classical twisted extensions in $C^*$-algebras/Steinberg algebras.
	
	\subsection{Groupoids associated to inverse semigroups} We now recall Steinberg's definitions and notation from \cite{S2} on the universal Paterson groupoid of germs associated to an inverse semigroup $S$. Note that we use germs in this paper, but we could have also used an ultrafilter approach, see \cite{ACaHJL}.

    We denote by $\widehat{E}(S)$ (this is called $\operatorname{Spec}(E(S))$ in \cite{S2}) the set of nonzero homomorphisms $\chi\colon E(S)\to \{0,1\}$. The sets $D_S(e):=\{\chi\in\widehat{E}(S) : \chi(e)=1\}$, $e\in E(S)$, are a basis of compact open subsets of $\widehat{E}(S)$. (We also write $D(e)$ if the inverse semigroup is understood.) For each $s\in S$ there is a homeomorphism $\beta_s\colon D_S(s^*s)\to D_S(ss^*)$ given by $\beta_s(\chi)(e)=\chi(s^*es)$. Then $s\mapsto \beta_s$ is a nondegenerate action of $S$ by partial homeomorphisms.
	
	The groupoid of germs is the set $\mathcal{G}(S):=(S\times \widehat{E}(S))/\sim$, where $(s,\chi)\sim (t,\lambda)$ if and only if $\chi=\lambda$ and there exists $u\le s,t$ with $\chi(u^*u)=1$. The groupoid $\mathcal{G}(S)$ has product $[s,\chi][t,\lambda]=[st,\lambda]$ if $\chi=\beta_t(\lambda)$; domain and range maps given by
	\[
	d_S([s,\chi])=\chi\quad\text{and}\quad r_S([s,\chi])=\beta_s(\chi);
	\]
	and inverse given by $[s,\chi]^{-1}=[s^*,\beta_s(\chi)]$. The groupoid $\mathcal{G}(S)$ has a basis of compact open bisections $D_S(s):=\{[s,\chi]:\chi\in D_S(s^*s)\}$, $s\in S$. Note that $\widehat{E}(S)\cong \mathcal{G}(S)^{(0)}$ with $\chi\mapsto [e,\chi]$ for $\chi(e)=1$.
	
	We will use the following results.
	
	\begin{lemma}\label{lem:resultaboutbeta}
		If $[s,\chi]=[t,\chi]$, then $\beta_s(\chi)=\beta_t(\chi)$.
	\end{lemma}
	
	\begin{proof}
		Let $u\le s,t$ with $\chi(u^*u)=1$, and let $e\in E(S)$. Then we have
		\[
		\beta_s(\chi)(e) = \chi(s^*es)= \chi(u^*u)\chi(s^*es)\chi(u^*u)= \chi((su^*u)^*esu^*u)= \chi(u^*eu)= \beta_u(\chi)(e).
		\]
		Hence $\beta_s(\chi)=\beta_u(\chi)$, and similarly we have $\beta_t(\chi)=\beta_u(\chi)$. The result follows.
	\end{proof}
	
	\begin{lemma}\label{lem:cliffconsequences}
		Let $N$ be a Clifford inverse semigroup, and let $\chi \in \widehat{E}(N)$.
		\begin{enumerate}
			\item If $a \in N$ and $\chi(a^*a)=1$, then $\beta_a(\chi)=\chi$.
			\item If $a,b \in N$ and $\chi(a^*a)=\chi(b^*b)=1$, then $[a,\chi]$ and $[b,\chi]$ are composable in $\mathcal{G}(N)$ and $[a,\chi][b,\chi]=[ab,\chi]$.
			\item If $a \in N$ and $\chi(a^*a)=1$, then $[a,\chi]^{-1}=[a^*,\chi]$.
		\end{enumerate}
	\end{lemma}
	
	\begin{proof}
		Since $N$ is Clifford, every idempotent is central, and $aa^*=a^*a$ for all $a\in N$. For (1), let $a\in N$ with $\chi(a^*a)=1$, and let $e\in E(N)$. Then
		\[
		\beta_a(\chi)(e)=\chi(a^*ea)=\chi(ea^*a)=\chi(e)\chi(a^*a)=\chi(e),
		\]
		and hence $\beta_a(\chi)=\chi$. For (2), just note that (1) implies that $\chi=\beta_b(\chi)$ and hence $([a,\chi],[b,\chi])\in\mathcal{G}(N)^{(2)}$. The formula for the product just comes from multiplication in $\mathcal{G}(N)$. For (3), we use (1) to get $[a,\chi]^{-1}=[a^*,\beta_a(\chi)]=[a^*,\chi]$.
	\end{proof}
	
	\subsection{Extensions of inverse semigroups} We recall the work of Lausch \cite{L} (see also \cite{DK} for a thorough account of all the details), which provides a connection between extensions of inverse semigroups and some cohomological invariants. Let
	\begin{equation}\label{eq:extensionINVSMGP}
		\xymatrix{N\ar[r]^{\iota}  & T \ar[r]^{\pi} & S\ar@/^8pt/ [l] ^{j}}
	\end{equation}
	be an extension of inverse semigroups. Recall that $\iota$ and $\pi$ are required to be idempotent-bijective, with $j:S\rightarrow T$ a section of $\pi$. Since $\pi$ is idempotent-bijective, we can assume that $j_{\vert E(S)}:E(S)\rightarrow E(T)$. Moreover, in this case $j_{\vert E(S)}=(\pi_{\vert E(T)})^{-1}$ is a homomorphism. We recall the following facts about this set up.
	\begin{enumerate}
		\item $N$ is a Clifford semigroup, and so idempotents are central, and $aa^*=a^*a$ for every $a\in N$.
		\item For each $t\in T$ there exists a unique $a_t\in N$ such that $t=\iota(a_t)j(\pi(t))$ and $\iota (a_ta_t^*)=j(\pi(t))j(\pi(t))^*$. We have $a_t=\iota^{-1}(tj(\pi(t))^*)$.
		\item The map $\alpha:=\iota^{-1}\circ j_{\vert E(S)}: E(S)\rightarrow E(N)$ is an isomorphism.
		\item Since $\pi(j(s)j(t))=\pi(j(st))$ for all $s,t\in S$, there exists a unique $f(s,t)\in N$ such that $j(s)j(t)=\iota(f(s,t))j(st)$. Moreover, $f(s,t)\in N_{\alpha(st(st)^*)}$, which is a group for any $st\in S$.
		\item If $\text{end} (N)$ denotes the relatively invertible endomorphisms of $N$ \cite[Section 8]{L} (see also \cite[Subsection 1.3]{DK}), then there is an action
		$$\lambda: S\rightarrow \text{end} (N)$$
		defined by the rule $\lambda_s(a):=\iota^{-1}(j(s)\iota (a)j(s)^*)$.
		\item There is a list of relations involving $\alpha, f$ and $\lambda$ (see \cite[Page 8]{DK}, or \cite[Proposition 3.11]{S2} for an extended version). In particular:
		\begin{enumerate}
			\item The map $f:S\times S\rightarrow N$ is a $2$-cocycle for the action $\lambda$: for any $s,t,r\in S$, we have $\lambda_s(f(t,r))f(s,tr)=f(s,t)f(st,r)$. This equation guarantees that the product defined in the inverse semigroup $N\ast_\Lambda S$ below is associative.
			\item For all $s,t\in S$ and for all $a\in N$, $(\lambda_s\circ \lambda_t)(a)=f(s,t)\lambda_{st}(a)f(s,t)^*$.
		\end{enumerate}
	\end{enumerate}
	
	We will need the following result.
	
	\begin{lemma}\label{lem:factsaboutf}
		Let $j$ be a section, and let $s,t\in S$.
		\begin{itemize}
			\item[(i)] We have $\iota(f(s,t))=j(s)j(t)j(st)^*$.
			\item[(ii)] If $j$ is order preserving, then
			\[
			p\le s,\, q\le t\implies f(p,q)\le f(s,t).
			\]
		\end{itemize}
	\end{lemma}
	
	\begin{proof}
		For (i), we use that $f(s,t)\in N_{\alpha(st(st)^*)}$ in the last implication to get
		\begin{align*}
			\iota(f(s,t))j(st)=j(s)j(t) &\implies \iota(f(s,t))j(st)j(st)^*=j(s)j(t)j(st)^*\\
			&\implies \iota(f(s,t))\iota(\alpha(st(st)^*))=j(s)j(t)j(st)^*\\
			&\implies \iota(f(s,t))=j(s)j(t)j(st)^*.
		\end{align*}
		
		For (ii), we let $p\le s$ and $q\le t$. Since $j$ is order preserving we have $j(p)\le j(s)$, $j(q)\le j(t)$, and then because multiplication and the inverse is order preserving we have $j(p)j(q)j(pq)^*\le j(s)j(t)j(st)^*$. Hence by (i) we have
		\[
		\iota(f(p,q))=j(p)j(q)j(pq)^*\le j(s)j(t)j(st)^*=\iota(f(s,t)),
		\]
		and hence $f(p,q)\le f(s,t)$.
	\end{proof}
	
	\begin{remark}\label{Remk: formal defined semigrp extension}
		A triple $\Lambda=(\alpha, \lambda, f)$ satisfying the list of properties in \cite[Page 8]{DK} is called a {\em twisted $S$-module}. To such a system we can define an inverse semigroup
		$$N\ast_\Lambda S:=\{(a,s)\in N\times S : aa^*=\alpha(ss^*)\}$$
		with product and inverse given by
		\[
		(a,s)(b,r):=(a\lambda_s(b)f(s,r),sr)\quad\text{and}\quad (a,s)^*:=(f(s^*, s)^*\lambda_{s^*}(a^*),s^*).
		\]
		Note that this means
		\begin{equation}\label{eq:rangeandsourceofas}
			(a,s)^*(a,s)=(\alpha(s^*s),s^*s)\quad\text{and}\quad (a,s)(a,s)^*=(\alpha(ss^*),ss^*).
		\end{equation}
		
		For $(b,p),(a,s)\in N\ast_\Lambda S$ we have
		\begin{equation}\label{eq:orderondecomp}
			(b,p)\le (a,s) \iff p\le s\text{ and }b\le \alpha(pp^*)a.
		\end{equation}
	\end{remark}
	
	\begin{remark}\label{rem:thevarphimap}
		Following Lausch \cite{L}, Dokuchaev and Khrypchenko \cite{DK}, we have that there is an isomorphism of inverse semigroups $\varphi\colon T\to N\ast_\Lambda S$ given by
		\[
		\varphi(t)=(\iota^{-1}(tj(\pi(t))^*),\pi(t)).
		\]
		Recall above that $a_t:=\iota^{-1}(tj(\pi(t))^*)$, and so we can write $\varphi(t)=(a_t,\pi(t))$.
		
		This isomorphism means we can deal with inverse semigroup extensions in terms of twisted crossed products of Clifford semigroups by ``arbitrary'' inverse semigroups (see, for example, \cite[Chapter 5]{Law} for a thorough account of inverse semigroup extensions).
	\end{remark}

	\section{Groupoid extensions with continuous global sections}\label{sec:withorderpreserving}
	
	Throughout this section we assume that
	\[
	\xymatrix{N\ar[r]^{\iota}  & T \ar[r]^{\pi} & S\ar@/^8pt/ [l] ^{j}}
	\]
	is an inverse semigroup extension with an {\em order-preserving} section $j:S\rightarrow T$, and $\Lambda=(\alpha,\lambda,f)$ is the  corresponding twisted $S$-module. In this section we prove the existence of a groupoid analogue $\widehat{\Lambda}=(\widehat{\alpha}, \widehat{\lambda},\widehat{f})$ of $\Lambda$, and an associated twisted crossed product groupoid $\mathcal{G}(N)\ast_{\widehat{\Lambda}}\mathcal{G}(S)$. In our main result we prove that $\mathcal{G}(T)$ is isomorphic to $\mathcal{G}(N)\ast_{\widehat{\Lambda}}\mathcal{G}(S)$, giving us a groupoid analogue of the isomorphism $T\cong N\ast_\Lambda S$.
	
	\begin{proposition}\label{prop:dualtriple}
		Suppose we have an extension of inverse semigroups as in \eqref{eq:extensionINVSMGP}, where the section $j\colon S\to T$ is order preserving. Let $\Lambda=(\alpha,\lambda,f)$ be the associated twisted $S$-module. There exists a triple $\widehat{\Lambda}:=(\widehat{\alpha}, \widehat{\lambda},\widehat{f})$ such that
		\begin{itemize}
			\item[(i)] $\widehat{\alpha}\colon \widehat{E}(S)\to \widehat{E}(N)$ given by $\widehat{\alpha}(\chi)= \chi\circ \alpha^{-1}$ is a homeomorphism,
			\item[(ii)] for each $[s,\chi]\in\mathcal{G}(S)$, there is a well-defined map $\widehat{\lambda}_{[s,\chi]}\colon \mathcal{G}(N)_{\widehat{\alpha}(\chi)}\to \mathcal{G}(N)_{\widehat{\alpha}(\beta_s(\chi))}$ given by
			\[
			\widehat{\lambda}_{[s,\chi]}([a,\widehat{\alpha}(\chi)])=[\lambda_s(a),\widehat{\alpha}(\beta_s(\chi))],
			\]
			\item[(iii)] the map $\widehat{f}\colon \mathcal{G}(S)^{(2)}\to \mathcal{G}(N)$ given by
			\[
			\widehat{f}([s,\beta_t(\chi)],[t,\chi])
			=
			[f(s,t),\widehat{\alpha}(\beta_{st}(\chi))]
			\]
			is well-defined and continuous,
			\item[(iv)] $[s,\chi]\mapsto \widehat{\lambda}_{[s,\chi]}$ is a well-defined continuous $\widehat{f}$-twisted action in the sense that
			\[
			\widehat{\lambda}_{g_1}\circ \widehat{\lambda}_{g_2} = \operatorname{Ad}_{\widehat{f}(g_1,g_2)}\circ \widehat{\lambda}_{g_1g_2}\quad\text{for all }(g_1,g_2)\in \mathcal{G}(S)^{(2)},
			\]
			and each $\widehat{\lambda}_{[s,\chi]}$ is a homeomorphism, and
			\item[(v)] $\widehat{f}$ is a $2$-cocycle for $\widehat{\lambda}$ in the sense that
			\[
			\widehat{\lambda}_{g_1}(\widehat{f}(g_2,g_3))\widehat{f}(g_1,g_2g_3)=\widehat{f}(g_1,g_2)\widehat{f}(g_1g_2,g_3)\quad\text{for all }(g_1,g_2,g_3)\in\mathcal{G}(S)^{(3)}.
			\]
		\end{itemize}
	\end{proposition}
	
	\begin{proof}
		Since $\alpha:E(S)\to E(N)$ is a semilattice isomorphism, $\widehat{\alpha}$ is well defined, and its inverse is given by $\widehat{\alpha}^{-1}(\psi)=\psi\circ \alpha$. For each $e\in E(N)$ we have
		\[
		\widehat{\alpha}^{-1}(D_N(e))
		=\{\chi\in \widehat{E}(S):(\chi\circ \alpha^{-1})(e)=1\}
		=D_S(\alpha^{-1}(e)),
		\]
		and so $\widehat{\alpha}$ is continuous. Similarly, for each $f\in E(S)$ we have $\widehat{\alpha}(D_S(f))=D_N(\alpha(f))$, and hence $\widehat{\alpha}$ is a homeomorphism. So (i) holds.
		
		For (ii), we first check that, for each $[s,\chi]\in \mathcal{G}(S)$ and $[a,\widehat{\alpha}(\chi)]\in \mathcal{G}(N)_{\widehat{\alpha}(\chi)}$, we have $[\lambda_s(a),\widehat{\alpha}(\beta_s(\chi))]\in \mathcal{G}(N)_{\widehat{\alpha}(\beta_s(\chi))}$. This follows because
		\begin{align*}
			\widehat{\alpha}(\beta_s(\chi))(\lambda_s(a)^*\lambda_s(a)) &= \widehat{\alpha}(\beta_s(\chi))(\lambda_s(a^*a)) \nonumber \\
			&= \beta_s(\chi)(\alpha^{-1}(\iota^{-1}(j(s)\iota(a^*a)j(s)^*))) \nonumber\\
			&= \beta_s(\chi)(s\alpha^{-1}(a^*a)s^*) \nonumber \\
			&=\chi(s^*s\alpha^{-1}(a^*a)s^*s) \nonumber \\
			&=\chi(s^*s)\widehat{\alpha}(\chi)(a^*a)\chi(s^*s) \nonumber \\
			&= 1.
		\end{align*}
		
		We now check that the map $\widehat{\lambda}_{[s,\chi]}$ is well defined. So assume $[a,\widehat{\alpha}(\chi)]=[b,\widehat{\alpha}(\chi)]$, and we need to prove that $[\lambda_s(a),\widehat{\alpha}(\beta_s(\chi))]=[\lambda_s(b),\widehat{\alpha}(\beta_s(\chi))]$. We know there exists
		$c\le a,b$ such that $\widehat{\alpha}(\chi)(c^*c)=\chi(\alpha^{-1}(c^*c))=1$. Since $j$ is order preserving, then so is $\lambda_s$, and we have $\lambda_s(c)\le \lambda_s(a),\lambda_s(b)$, and we claim that $\widehat{\alpha}(\beta_s(\chi))(\lambda_s(c)^*\lambda_s(c))=1$. To see this, first note that for $e\in E(N)$ we have
		\[
		\alpha^{-1}(\lambda_s(e))=\pi|_{E(T)}\circ\iota(\iota^{-1}(j(s)\iota(e)j(s)^*)) =s\alpha^{-1}(e)s^*.
		\]
		Hence
		\begin{align*}
			\widehat{\alpha}(\beta_s(\chi))(\lambda_s(c)^*\lambda_s(c))&=\chi(s^*\alpha^{-1}(\lambda_s(c^*c))s)\\
			&=\chi(s^*s\alpha^{-1}(c^*c)s^*s)\\
			&=\chi(s^*s)\chi(\alpha^{-1}(c^*c))\\
			&=1,
		\end{align*}
		and the claim holds. Hence $[\lambda_s(a),\widehat{\alpha}(\beta_s(\chi))]=[\lambda_s(b),\widehat{\alpha}(\beta_s(\chi))]$, and so $\widehat{\lambda}_{[s,\chi]}$ is well defined.
		
		For (iii), first note that since $f(s,t)^*f(s,t)=\alpha(st(st)^*)$, it follows that
		\[
		\widehat{\alpha}(\beta_{st}(\chi))(f(s,t)^*f(s,t))=\widehat{\alpha}(\beta_{st}(\chi))(\alpha(st(st)^*))=\beta_{st}(\chi)(st(st)^*)=\chi((st)^*st)=1.
		\]
		We now check that the formula for $\widehat{f}$ is well defined. Suppose $[t_1,\chi]=[t_2,\chi]$ and $[s_1,\beta_{t_1}(\chi)]=[s_2,\beta_{t_2}(\chi)]$. We need to prove that
		\[
		[f(s_1,t_1),\widehat{\alpha}(\beta_{s_1t_1}(\chi))]=[f(s_2,t_2),\widehat{\alpha}(\beta_{s_2t_2}(\chi))].
		\]
		Since $[s_1t_1,\chi]=[s_2t_2,\chi]$ we know from Lemma~\ref{lem:resultaboutbeta} that $\beta_{s_1t_1}(\chi)=\beta_{s_2t_2}(\chi)$. Hence $\widehat{\alpha}(\beta_{s_1t_1}(\chi))=\widehat{\alpha}(\beta_{s_2t_2}(\chi))$.
		
		Now let $q\le t_1,t_2$ with $\chi(q^*q)=1$, and $p\le s_1,s_2$ with $\beta_{t_1}(\chi)(p^*p)=1$. We claim that $f(p,q)\le f(s_1,t_1),f(s_2,t_2)$ and
		\[
		\widehat{\alpha}(\beta_{s_1t_1}(\chi))(f(p,q)^*f(p,q))=1.
		\]
		The first assertion follows from Lemma~\ref{lem:factsaboutf}(ii). For the second, first note that since $f(p,q)\in N_{\alpha(pq(pq)^*)}$ we have $f(p,q)^*f(p,q)=\alpha(pq(pq)^*)$, and hence
		\[
		\widehat{\alpha}(\beta_{s_1t_1}(\chi))(f(p,q)^*f(p,q))=\beta_{s_1t_1}(\chi)(pq(pq)^*).
		\]
		Using the inverse semigroup identity $x\le y\implies y^*x=x^*x$ we get
		\begin{align*}
			\beta_{s_1t_1}(\chi)(pq(pq)^*) &= \chi((s_1t_1)^*pq(pq)^*s_1t_1)\\
			&= \chi((pq)^*pq)\\
			&=\chi((t_1q^*q)^*p^*pt_1q^*q)\\
			&=\chi(q^*q)\chi(t_1^*p^*pt_1)\chi(q^*q)\\
			&=\chi(q^*q)\beta_{t_1}(\chi)(p^*p)\chi(q^*q)\\
			&=1,
		\end{align*}
		and so the claim holds. Hence $\widehat{f}([s,\beta_t(\chi)],[t,\chi])=[f(s,t),\widehat{\alpha}(\beta_{st}(\chi))]$ is well defined.
		
		We now prove that $\hat f$ is continuous. Fix $s,t\in S$. Note that on the set
		\[
		\bigl(D_S(s)\times D_S(t)\bigr)\cap \GrpG(S)^{(2)}
		= \{([s,\beta_t(\chi)],[t,\chi]) : \chi\in D_S((st)^*st)\},
		\]
		$\widehat{f}$ is the composition
		\[
		([s,\beta_t(\chi)],[t,\chi]) \mapsto \chi \mapsto \beta_{st}(\chi) \mapsto \widehat{\alpha}(\beta_{st}(\chi)) \mapsto [f(s,t),\widehat{\alpha}(\beta_{st}(\chi))]
		\]
		which are all continuous. Hence $\widehat{f}$ is continuous on $\bigl(D_S(s)\times D_S(t)\bigr)\cap \GrpG(S)^{(2)}$, and since these sets are a basis for $\GrpG(S)^{(2)}$, the continuity of $\widehat{f}$ follows.
		
		To prove (iv), we first check that if $[s,\chi]=[t,\chi]$, then $\widehat{\lambda}_{[s,\chi]}=\widehat{\lambda}_{[t,\chi]}$. Let $[a,\widehat{\alpha}(\chi)]\in \mathcal{G}(N)_{\widehat{\alpha}(\chi)}$. We need to show that $[\lambda_s(a),\widehat{\alpha}(\beta_s(\chi))]=[\lambda_t(a),\widehat{\alpha}(\beta_t(\chi))]$. To see this, first let $u\le s,t$ with $\chi(u^*u)=1$. Then $u=su^*u=tu^*u$, and for each $e\in E(N)$ we have
		\[
		\beta_u(\chi)(e)=\chi(u^*eu)=\chi((su^*u)^*esu^*u)=\chi(u^*u)\chi(s^*es)\chi(u^*u)=\beta_s(\chi)(e).
		\]
		Hence $\beta_s(\chi)=\beta_u(\chi)$, which is also $\beta_t(\chi)$, and so $\widehat{\alpha}(\beta_s(\chi))=\widehat{\alpha}(\beta_t(\chi))$. Now to finish the proof that $[\lambda_s(a),\widehat{\alpha}(\beta_s(\chi))]=[\lambda_t(a),\widehat{\alpha}(\beta_t(\chi))]$ we claim that $v:=\lambda_u(a\alpha(u^*u))$ satisfies
		\[
		v\le \lambda_s(a),\lambda_t(a)\quad\text{and}\quad \widehat{\alpha}(\beta_s(\chi))(v^*v)=1.
		\]
		To see that $v\le \lambda_s(a)$ we first note that since $j$ is order preserving, we have $j(u)=j(su^*u)=j(s)j(u^*u)$ by \cite{S2}. Then
		\begin{align*}
			\iota(v) &= j(u)\iota(a\alpha(u^*u))j(u)^*\\
			&=j(s)j(u^*u)\iota(a\alpha(u^*u))j(u^*u)j(s)^*\\
			&=j(s)\iota(\alpha(u^*u))\iota(a\alpha(u^*u))\iota(\alpha(u^*u))j(s)^*\\
			&=j(s)\iota(a\alpha(u^*u))j(s)^*,
		\end{align*}
		which means that $v=\lambda_s(a\alpha(u^*u))$. But since $a\alpha(u^*u)\le a$ and $\lambda_s$ is order preserving, this means $v\le \lambda_s(a)$. A similar argument gives $v\le \lambda_t(a)$.
		
		Now, for any $e\in E(N)$ we use that $\chi(u^*u)=1$ to get
		\begin{align*}
			\widehat{\alpha}(\beta_u(\chi))(\lambda_u(e))=\beta_u(\chi)(\alpha^{-1}(\lambda_u(e)))=\beta_u(\chi)(u\alpha^{-1}(e)u^*)&=\chi(u^*u\alpha^{-1}(e)u^*u)\\
			&=\chi(\alpha^{-1}(e))\\
			&=\widehat{\alpha}(\chi)(e).
		\end{align*}
		Using this identity with $e=\alpha(u^*u)a^*a\alpha(u^*u)$ we get
		\begin{align*}
			\widehat{\alpha}(\beta_s(\chi))(v^*v) &= \widehat{\alpha}(\beta_u(\chi))(v^*v) \\
			&= \widehat{\alpha}(\beta_u(\chi))(\lambda_u(\alpha(u^*u)a^*a\alpha(u^*u))) \\
			&= \widehat{\alpha}(\chi)(\alpha(u^*u)a^*a\alpha(u^*u))\\
			&=\chi(u^*u)\widehat{\alpha}(\chi)(a^*a)\chi(u^*u)\\
			&=1.
		\end{align*}
		This proves the claim, and so we see that $[s,\chi]\mapsto \widehat{\lambda}_{[s,\chi]}$ is well defined.
		
		Next, observe that the restriction of
		\[
		\widehat\lambda:\{([s,\chi],[a,\widehat{\alpha}(\chi)]) : s\in S,a\in N, \chi\in D_S(s^*s)\cap D_S(\alpha^{-1}(a^*a))\} \to \GrpG(N)
		\]
		given by $\widehat{\lambda}([s,\chi],[a,\widehat{\alpha}(\chi)])=\widehat{\lambda}_{[s,\chi]}([a,\widehat{\alpha}(\chi)])$ to the set
		\[
		Z_{s,a}:= \{([s,\chi],[a,\widehat{\alpha}(\chi)]) : \chi\in D_S(s^*s)\cap D_S(\alpha^{-1}(a^*a))\}
		\]
		is the composition of the continuous maps
		\[
		([s,\chi],[a,\widehat{\alpha}(\chi)])\mapsto \chi\mapsto \beta_s(\chi)\mapsto \widehat{\alpha}(\beta_s(\chi)) \mapsto [\lambda_s(a),\widehat{\alpha}(\beta_s(\chi))].
		\]
		Since the sets $Z_{s,a}$ are a basis of compact open sets for the domain of $\widehat{\lambda}$, we see that $\widehat{\lambda}$ is continuous.
		
		We now check that $\widehat{\lambda}$ is $\widehat{f}$-twisted. Let $g_1=[s,\beta_t(\chi)]$ and $g_2=[t,\chi]$ be composable elements of $\mathcal{G}(S)$. Then
		\begin{align}
			\widehat{\lambda}_{g_1}\bigl(\widehat{\lambda}_{g_2}([a,\widehat{\alpha}(\chi)])\bigr)
			&=
			\widehat{\lambda}_{[s,\beta_t(\chi)]}\bigl([\lambda_t(a),\widehat{\alpha}(\beta_t(\chi))]\bigr)\nonumber \\
			&=
			[\lambda_s(\lambda_t(a)),\widehat{\alpha}(\beta_s\beta_t(\chi))] \nonumber \\
			&=
			[\lambda_s(\lambda_t(a)),\widehat{\alpha}(\beta_{st}(\chi))]. \label{eq:1forlambdatwistedaction}
		\end{align}
		By the identity $\lambda_s\circ\lambda_t = \operatorname{Ad}_{f(s,t)}\circ \lambda_{st}$, this is
		\[
		[f(s,t)\lambda_{st}(a)f(s,t)^*,\widehat{\alpha}(\beta_{st}(\chi))].
		\]
		On the other hand, we have
		\[
		\widehat{f}(g_1,g_2)=[f(s,t),\widehat{\alpha}(\beta_{st}(\chi))]
		\]
		and
		\[
		\widehat{\lambda}_{g_1g_2}([a,\widehat{\alpha}(\chi)])
		=
		\widehat{\lambda}_{[st,\chi]}([a,\widehat{\alpha}(\chi)])
		=
		[\lambda_{st}(a),\widehat{\alpha}(\beta_{st}(\chi))].
		\]
		Hence
		\begin{align}
			\operatorname{Ad}_{\widehat{f}(g_1,g_2)}\bigl(\widehat{\lambda}_{g_1g_2}([a,\widehat{\alpha}(\chi)])\bigr)
			&=
			\widehat{f}(g_1,g_2)\,\widehat{\lambda}_{g_1g_2}([a,\widehat{\alpha}(\chi)])\,\widehat{f}(g_1,g_2)^{-1} \nonumber \\
			&=
			[f(s,t),\widehat{\alpha}(\beta_{st}(\chi))]
			[\lambda_{st}(a),\widehat{\alpha}(\beta_{st}(\chi))]
			[f(s,t),\widehat{\alpha}(\beta_{st}(\chi))]^{-1}.\label{eq:2forlambdatwistedaction}
		\end{align}
		To see that \eqref{eq:1forlambdatwistedaction} and \eqref{eq:2forlambdatwistedaction} agree, first note that for $\eta=\widehat{\alpha}(\beta_{st}(\chi))$ we have already seen that
		\[
		\eta(f(s,t)^*f(s,t))=1\quad\text{and}\quad \eta(\lambda_{st}(a)^*\lambda_{st}(a))=1
		\]
		(in the proofs of (iii) and (ii), respectively). This means we can apply Lemma~\ref{lem:cliffconsequences} to continue the calculation giving \eqref{eq:2forlambdatwistedaction} to get
		\[
		\operatorname{Ad}_{\widehat{f}(g_1,g_2)}\bigl(\widehat{\lambda}_{g_1g_2}([a,\widehat{\alpha}(\chi)])\bigr)=
		[f(s,t),\eta][\lambda_{st}(a),\eta][f(s,t)^*,\eta]=[f(s,t)\lambda_{st}(a)f(s,t)^*,\eta],
		\]
		which is the right-hand side of \eqref{eq:1forlambdatwistedaction}. Hence
		\[
		\widehat{\lambda}_{g_1}\circ \widehat{\lambda}_{g_2}
		=
		\operatorname{Ad}_{\widehat{f}(g_1,g_2)} \circ \widehat{\lambda}_{g_1g_2}.
		\]
		
		We now show that $\widehat{\lambda}_{[e,\chi]}$ is trivial for each $[e,\chi]\in\mathcal{G}(S)^{(0)}$. So we show that we have $\widehat{\lambda}_{[e,\chi]}([a,\widehat{\alpha}(\chi)])=[a,\widehat{\alpha}(\chi)]$ for all $[a,\widehat{\alpha}(\chi)]\in\mathcal{G}(N)_{\widehat{\alpha}(\chi)}$. Note that we have $\chi(e)=1$, and it follows that $\beta_e(\chi)=\chi$. For $a\in N$ we also have $\lambda_e(a)=\alpha(e)a\le a$, and hence
		\[
		\widehat{\alpha}(\chi)(\lambda_e(a)^*\lambda_e(a))=\widehat{\alpha}(\chi)(\alpha(e)a^*a)=\chi(e)\widehat{\alpha}(\chi)(a^*a)=1.
		\]
		This means $[\lambda_e(a),\widehat{\alpha}(\chi)]=[a,\widehat{\alpha}(\chi)]$, and so we have
		\[
		\widehat{\lambda}_{[e,\chi]}([a,\widehat{\alpha}(\chi)])=[\lambda_e(a),\widehat{\alpha}(\beta_e(\chi))]=[a,\widehat{\alpha}(\chi)]
		\]
		So $\widehat{\lambda}_{[e,\chi]}$ acts as the identity on $\mathcal{G}(N)_{\widehat{\alpha}(\chi)}$.
		
		We now show that $\widehat\lambda_{[s,\chi]}$ is a bijection with inverse $\widehat\lambda_{[s^*,\beta_s(\chi)]}$. First note that
		\begin{align*}
			\lambda_{s^*}(\lambda_s(a))=\lambda_{s^*}(\iota^{-1}(j(s)\iota(a)j(s)^*))&=\iota^{-1}(j(s^*)j(s)\iota(a)j(s)^*j(s^*)^*)\\
			&=\alpha(s^*s)a\alpha(s^*s)\\
			&=\alpha(s^*s)a,
		\end{align*}
		where the last equality holds because idempotents in $N$ are central. So for $[a,\widehat{\alpha}(\chi)]\in \GrpG(N)_{\widehat{\alpha}(\chi)}$ we have
		\[
		\widehat\lambda_{[s^*,\beta_s(\chi)]}(\widehat\lambda_{[s,\chi]}([a,\widehat{\alpha}(\chi)]))=[\lambda_{s^*}(\lambda_s(a)),\widehat{\alpha}(\chi)]=[\alpha(s^*s)a,\widehat{\alpha}(\chi)].
		\]
		Since $\alpha(s^*s)a\le a$ and
		\[
		\widehat\alpha(\chi)\bigl((\alpha(s^*s)a)^*\alpha(s^*s)a\bigr)
		= \widehat\alpha(\chi)\bigl(a^*\alpha(s^*s)a\bigr)
		= \widehat\alpha(\chi)\bigl(\alpha(s^*s)a^*a\bigr)
		= \chi(s^*s)\widehat\alpha(\chi)(a^*a)
		= 1,
		\]
		we have
		\[
		\widehat\lambda_{[s^*,\beta_s(\chi)]}(\widehat\lambda_{[s,\chi]}([a,\widehat{\alpha}(\chi)]))=[a,\widehat{\alpha}(\chi)].
		\]
		A similar argument shows that $\widehat{\lambda}_{[s,\chi]}\circ \widehat{\lambda}_{[s^*,\beta_s(\chi)]}$ is the identity on $\GrpG(N)_{\widehat{\alpha}(\beta_s(\chi))}$, and hence $\widehat{\lambda}_{[s,\chi]}$ is bijective.
		
		Finally, both $\hat\lambda_g$ and $\hat\lambda_{g^{-1}}$ are continuous, since they are
		restrictions of the continuous map $\hat\lambda$. Therefore $\hat\lambda_g$ is a
		homeomorphism for every $g\in \GrpG(S)$.

		For (v), let $(g_1,g_2,g_3)\in\mathcal{G}(S)^{(3)}$ with
		\[
		g_1=[r,\beta_{st}(\chi)],\, g_2=[s,\beta_t(\chi)] ,\text{ and }g_3=[t,\chi].
		\]
		Then
		\begin{align*}
			\widehat{\lambda}_{g_1}(\widehat{f}(g_2,g_3))\widehat{f}(g_1,g_2g_3) &= \widehat{\lambda}_{[r,\beta_{st}(\chi)]}
			\big(\widehat{f}([s,\beta_t(\chi)],[t,\chi])\big)\,
			\widehat{f}([r,\beta_{st}(\chi)],[st,\chi]) \\
			&=
			[\lambda_r(f(s,t)),\widehat{\alpha}(\beta_{rst}(\chi))]\,
			[f(r,st),\widehat{\alpha}(\beta_{rst}(\chi))] \\
			&=
			[\lambda_r(f(s,t))f(r,st),\widehat{\alpha}(\beta_{rst}(\chi))],
		\end{align*}
		where we are again applying Lemma~\ref{lem:cliffconsequences} to get the final equality. On the other hand, using Lemma~\ref{lem:cliffconsequences} also gives
		\begin{align*}
			\widehat{f}(g_1,g_2)\widehat{f}(g_1g_2,g_3)&=\widehat{f}([r,\beta_s\beta_t(\chi)],[s,\beta_t(\chi)])\,
			\widehat{f}([rs,\beta_t(\chi)],[t,\chi]) \\
			&=
			[f(r,s),\widehat{\alpha}(\beta_{rst}(\chi))]\,
			[f(rs,t),\widehat{\alpha}(\beta_{rst}(\chi))] \\
			&=
			[f(r,s)f(rs,t),\widehat{\alpha}(\beta_{rst}(\chi))].
		\end{align*}
		Since $f$ satisfies the cocycle identity
		\[
		\lambda_r(f(s,t))f(r,st)=f(r,s)f(rs,t),
		\]
		(v) follows.
	\end{proof}
	
	\begin{definition}\label{def:gpoidtwistedaction}
		We call a triple $\widehat{\Lambda}=(\widehat{\alpha},\widehat{\lambda},\widehat{f})$ satisfying the properties outlined in Proposition~\ref{prop:dualtriple} a {\em groupoid twisted action}.
	\end{definition}
	
	\begin{proposition}\label{prop:twistedgroupoid}
		Suppose we have an extension of inverse semigroups as in \eqref{eq:extensionINVSMGP}, where the section $j\colon S\to T$ is order preserving. Let $\Lambda=(\alpha,\lambda,f)$ be the associated twisted $S$-module, and $\widehat{\Lambda}:=(\widehat{\alpha}, \widehat{\lambda},\widehat{f})$ the groupoid twisted action from Proposition~\ref{prop:dualtriple}. The set
		\begin{align*}
			\mathcal{G}(N)\ast_{\widehat{\Lambda}}\mathcal{G}(S)
			&:=
			\{(g,h)\in \mathcal{G}(N)\times \mathcal{G}(S): d_N(g)=\widehat{\alpha}(r_S(h))\}\\
			&=\left\{
			\left([a,\widehat{\alpha}(\beta_s(\chi))],[s,\chi]\right)
			\in \mathcal{G}(N)\times \mathcal{G}(S)
			: (a,s)\in N\ast_{\Lambda}S
			\right\}
		\end{align*}
		is an ample groupoid with unit space
		\[
		(\mathcal{G}(N)\ast_{\widehat{\Lambda}}\mathcal{G}(S))^{(0)}=\{([f,\widehat{\alpha}(\chi)],[e,\chi]) : \chi\in \widehat{E}(S),\, \chi(e)=1=\widehat{\alpha}(\chi)(f)\},
		\]
		which we identify with $\widehat{E}(S)$ via $([f,\widehat{\alpha}(\chi)],[e,\chi])\mapsto \chi$; domain and range maps given by
		\[
		d(g,h)=d_S(h)\quad\text{and}\quad r(g,h)=r_S(h);
		\]
		product given by
		\[
		(g_1,h_1)(g_2,h_2)
		=
		\bigl(g_1\,\widehat{\lambda}_{h_1}(g_2)\,\widehat{f}(h_1,h_2),\,h_1h_2\bigr)
		\]
		for a composable pair $(h_1,h_2)\in \mathcal{G}(S)^{(2)}$; and inverse
		\[
		(g,h)^{-1}
		=
		\bigl(\widehat{f}(h^{-1},h)^{-1}\,\widehat{\lambda}_{h^{-1}}(g^{-1}),\,h^{-1}\bigr).
		\]
		The collection
		\[
		\big\{\mathcal{D}_{a,s}
		:=
		\left\{
		\left([a,\widehat{\alpha}(\beta_s(\chi))],[s,\chi]\right)
		: \chi\in D_S(s^*s)
		\right\} : (a,s)\in N\ast_\Lambda S\big\}
		\]
        is a basis of compact open bisections for $\mathcal{G}(N)\ast_{\widehat{\Lambda}}\mathcal{G}(S)$, and the topology generated by this basis is the relative topology inherited from $\mathcal{G}(N) \times \mathcal{G}(S)$.
	\end{proposition}
	
	\begin{proof}
		It is routine to check the well-definedness of the structure maps, closure of multiplication and inversion, and the unit/inverse domain-range identities. For associativity, let $((g_1,h_1),(g_2,h_2),(g_3,h_3))\in (\GrpG(N)\ast_{\widehat{\Lambda}}\GrpG(S))^{(3)}$. The only nontrivial check is on the first coordinate: we need to show that
		\[
		g_1\widehat{\lambda}_{h_1}(g_2)\widehat{f}(h_1,h_2)\widehat{\lambda}_{h_1h_2}(g_3)\widehat{f}(h_1h_2,h_3)
		=
		g_1\widehat{\lambda}_{h_1}(g_2\widehat{\lambda}_{h_2}(g_3)\widehat{f}(h_2,h_3))\widehat{f}(h_1,h_2h_3).
		\]
		Ignoring the $g_1$ common to the beginning of both, we use the identities in (iv) and (v) of Proposition~\ref{prop:dualtriple} to get
		\begin{align*}
			\widehat{\lambda}_{h_1}(g_2\widehat{\lambda}_{h_2}(g_3)\widehat{f}(h_2,h_3))&\widehat{f}(h_1,h_2h_3)\\
			&= \widehat{\lambda}_{h_1}(g_2)\widehat{\lambda}_{h_1}(\widehat{\lambda}_{h_2}(g_3))\widehat{\lambda}_{h_1}(\widehat{f}(h_2,h_3))\widehat{f}(h_1,h_2h_3)\\
			&= \widehat{\lambda}_{h_1}(g_2)\widehat{\lambda}_{h_1}(\widehat{\lambda}_{h_2}(g_3))\widehat{f}(h_1,h_2)\widehat{f}(h_1h_2,h_3)\\
			&= \widehat{\lambda}_{h_1}(g_2) \widehat{f}(h_1,h_2)\widehat{\lambda}_{h_1h_2}(g_3)\widehat{f}(h_1,h_2)^{-1} \widehat{f}(h_1,h_2)\widehat{f}(h_1h_2,h_3)\\
			&=\widehat{\lambda}_{h_1}(g_2)\widehat{f}(h_1,h_2)\widehat{\lambda}_{h_1h_2}(g_3)\widehat{f}(h_1h_2,h_3),
		\end{align*}
		and associativity follows. So $\GrpG(N)*_{\widehat{\Lambda}}\GrpG(S)$ is a groupoid. Since
		\[
		\mathcal{D}_{a,s}=(D_N(a)\times D_S(s))\cap (\GrpG(N)\ast_{\widehat{\Lambda}}\GrpG(S))
		\]
		for each $(a,s)\in N\ast_\Lambda S$, these sets are a basis for the relative topology on $\GrpG(N)*_{\widehat{\Lambda}}\GrpG(S)$ inherited as a subspace of $\GrpG(N) \times \GrpG(S)$; it follows that $\GrpG(N)\ast_{\widehat{\Lambda}}\GrpG(S)$ is an ample groupoid and the set of $\mathcal{D}_{a,s}$ is a basis of compact open bisections.
	\end{proof}
	
	\begin{theorem}\label{thm:mainthmwithorderpreserving}
		Suppose we have an extension of inverse semigroups as in \eqref{eq:extensionINVSMGP}, where the section $j\colon S\to T$ is order preserving. Let $\Lambda=(\alpha,\lambda,f)$ be the associated twisted $S$-module, and $\widehat{\Lambda}:=(\widehat{\alpha}, \widehat{\lambda},\widehat{f})$ the groupoid twisted action from Proposition~\ref{prop:dualtriple}. The map $\widehat{\varphi}\colon \mathcal{G}(T) \to \mathcal{G}(N)\ast_{\widehat{\Lambda}}\mathcal{G}(S)$ given by
		\[
		\widehat{\varphi}([t,\chi])
		=
		\bigl([a_t,\widehat{\alpha}(\beta_{\pi(t)}(\chi\circ j|_{E(S)}))],[\pi(t),\chi\circ j|_{E(S)}]\bigr),
		\]
		is a well-defined isomorphism of topological groupoids.
	\end{theorem}	
	
	\begin{proof}
		
		Since $\varphi:T\to N\ast_{\Lambda}S$ is an isomorphism of inverse semigroups,
		\cite[Theorem~3.4 and Proposition~3.5]{S2} imply that the
		induced map
		\[
		\GrpG(\varphi):\GrpG(T)\to \GrpG(N\ast_{\Lambda}S)
		\]
		given by
		\[
		\GrpG(\varphi)([t,\chi])
		=
		[\varphi(t),\,\chi\circ (\varphi|_{E(T)})^{-1}]
		\]
		is an isomorphism of topological groupoids. We will compose $\GrpG(\varphi)$ with an isomorphism from $\GrpG(N\ast_\Lambda S)$ to $\GrpG(N)\ast_{\widehat{\Lambda}}\GrpG(S)$ to get our desired map $\widehat{\varphi}$.
		
		We claim that $\psi\colon \GrpG(N\ast_\Lambda S)\to \GrpG(N)\ast_{\widehat{\Lambda}}\GrpG(S)$ given by
		\begin{equation}\label{eq:defofpsi}
			\psi([(a,s),\chi]) = ([a,\widehat{\alpha}(\beta_s(\widetilde{\chi}))],[s,\widetilde{\chi}]),
		\end{equation}
		where $\widetilde{\chi}\in \widehat{E}(S)$ is given by $\widetilde{\chi}(e):=\chi(\alpha(e),e)$, is a well-defined isomorphism of topological groupoids.
		
		We first show that $\psi$ is well defined. Suppose $[(a,s),\chi]=[(b,t),\chi]$. We need to show that
		\[
		[a,\widehat{\alpha}(\beta_s(\widetilde{\chi}))]=[b,\widehat{\alpha}(\beta_t(\widetilde{\chi}))]\quad\text{and}\quad [s,\widetilde{\chi}]=[t,\widetilde{\chi}].
		\]
		Since $[(a,s),\chi]=[(b,t),\chi]$ we know that there exists $(c,p)\in N\ast_\Lambda S$ with $(c,p)\le (a,s),(b,t)$ and $\chi((c,p)^*(c,p))=1$. From \eqref{eq:orderondecomp}, this implies that $p\le s,t$ and also that $c\le \alpha(pp^*)a,\alpha(pp^*)b$. From \eqref{eq:rangeandsourceofas} we know that $(c,p)^*(c,p)=(\alpha(p^*p),p^*p)$, and hence the latter of our two desired identities holds:
		\[
		\widetilde{\chi}(p^*p)=\chi(\alpha(p^*p),p^*p)=\chi((c,p)^*(c,p))=1,
		\]
		giving $[s,\widetilde{\chi}]=[t,\widetilde{\chi}]$. For the former identity, first note that $\beta_s(\widetilde{\chi})=r_S([s,\widetilde{\chi}])=r_S([t,\widetilde{\chi}])=\beta_t(\widetilde{\chi})$, and hence we have $\widehat{\alpha}(\beta_s(\widetilde{\chi}))=\widehat{\alpha}(\beta_t(\widetilde{\chi}))$. We now use that $\alpha^{-1}(c^*c)=p^*p=pp^*$ and that $s^*pp^*s=p^*p$ to get
		\[
		\widehat{\alpha}(\beta_s(\widetilde{\chi}))(c^*c) = \beta_s(\widetilde{\chi})(\alpha^{-1}(c^*c))=\widetilde{\chi}(s^*pp^*s)=\widetilde{\chi}(p^*p)=1.
		\]
		This shows that $[a,\widehat{\alpha}(\beta_s(\widetilde{\chi}))]=[b,\widehat{\alpha}(\beta_t(\widetilde{\chi}))]$, and hence $\psi$ is well defined.
		
		We now show that $\psi$ is multiplicative. So we claim that
		\begin{equation}\label{eq:psiismultiplicative}
			\psi([(a,s),\beta_{(b,t)}(\chi)])\psi([(b,t),\chi]) = \psi([(a,s)(b,t),\chi]).
		\end{equation}
		The right-hand side of \eqref{eq:psiismultiplicative} is
		\begin{equation}\label{eq:RHSforpsimulti}
			\psi([(a,s)(b,t),\chi])=\psi([(a\lambda_s(b)f(s,t),st),\chi])=([a\lambda_s(b)f(s,t),\widehat{\alpha}(\beta_{st}(\widetilde{\chi}))],[st,\widetilde{\chi}]).
		\end{equation}
		For the left-hand side, we first make the following observation. For an idempotent $e\in S$ we must have $(b,t)^*(\alpha(e),e)(b,t)=(\alpha(t^*et),t^*et)$ because the second coordinate is $t^*et$. Hence we have
		\[
		\widetilde{\beta_{(b,t)}(\chi)}(e) = \beta_{(b,t)}(\chi)(\alpha(e),e)= \chi((b,t)^*(\alpha(e),e)(b,t))= \chi(\alpha(t^*et),t^*et)=\beta_t(\widetilde{\chi})(e),
		\]
		and so $\widetilde{\beta_{(b,t)}(\chi)}=\beta_t(\widetilde{\chi})$. We now calculate the left-hand side of \eqref{eq:psiismultiplicative} to be
		\begin{align*}
			&\psi([(a,s),\beta_{(b,t)}(\chi)])\psi([(b,t),\chi]) \\
			&= \big([a,\widehat{\alpha}(\beta_s(\beta_t(\widetilde{\chi})))],[s,\beta_t(\widetilde{\chi})]\big)\big([b,\widehat{\alpha}(\beta_t(\widetilde{\chi}))],[t,\widetilde{\chi}]\big)\\
			&= \big( [a,\widehat{\alpha}(\beta_{st}(\widetilde{\chi}))]\widehat{\lambda}_{[s,\beta_t(\widetilde{\chi})]}([b,\widehat{\alpha}(\beta_t(\widetilde{\chi}))])\widehat{f}([s,\beta_t(\widetilde{\chi})],[t,\widetilde{\chi}]) , [s,\beta_t(\widetilde{\chi})][t,\widetilde{\chi}]\big).
		\end{align*}
		The second component is just $[st,\widetilde{\chi}]$, which matches the second component in \eqref{eq:RHSforpsimulti}. For the first component we use Lemma~\ref{lem:cliffconsequences} to get
		\begin{align*}
			[a,\widehat{\alpha}(\beta_{st}(\widetilde{\chi}))]\widehat{\lambda}_{[s,\beta_t(\widetilde{\chi})]}&([b,\widehat{\alpha}(\beta_t(\widetilde{\chi}))])\widehat{f}([s,\beta_t(\widetilde{\chi})],[t,\widetilde{\chi}])\\  &=
			[a,\widehat{\alpha}(\beta_{st}(\widetilde{\chi}))][\lambda_s(b),\widehat{\alpha}(\beta_{st}(\widetilde{\chi}))][f(s,t),\widehat{\alpha}(\beta_{st}(\widetilde{\chi}))]\\
			&= [a\lambda_s(b)f(s,t),\widehat{\alpha}(\beta_{st}(\widetilde{\chi}))],
		\end{align*}
		which is the corresponding component in \eqref{eq:RHSforpsimulti}. It follows that $\psi$ is multiplicative.
		
		To see that $\psi$ is bijective, we build its inverse $\Psi$. We claim that $\Psi\colon \GrpG(N)\ast_{\widehat{\Lambda}}\GrpG(S)\to \GrpG(N\ast_\Lambda S)$ given by
		\begin{equation}\label{eq:defofPsi}
			\Psi([a,\widehat{\alpha}(\beta_s(\widetilde{\eta}))],[s,\widetilde{\eta}])=[(a,s),\eta]
		\end{equation}
		is well defined. The proof is similar to the well-definedness of $\psi$: if we assume that
		\[
		[a,\widehat{\alpha}(\beta_s(\widetilde{\eta}))]=[b,\widehat{\alpha}(\beta_t(\widetilde{\eta}))]\quad\text{and}\quad [s,\widetilde{\eta}]=[t,\widetilde{\eta}],
		\]
		then we can find $(\alpha(pp^*)c,p)$ with $(\alpha(pp^*)c,p)\le (a,s),(b,t)$ and
		\[
		\eta((\alpha(pp^*)c,p)^*(\alpha(pp^*)c,p))=1,
		\]
		and hence $[(a,s),\eta]=[(b,t),\eta]$. So $\Psi$ is well defined.
		
		It is immediate from the formulas in \eqref{eq:defofpsi} and \eqref{eq:defofPsi} that the compositions $\psi\circ\Psi$ and $\Psi\circ\psi$ are the identity on their respective groupoids. So $\psi$ is bijective.
		
		It remains to show that $\psi$ is continuous and open. Recall that a basis of
		compact open bisections for $\GrpG(N\ast_{\Lambda}S)$ is
		\[
		\big\{D_{N\ast_{\Lambda}S}(a,s)
		:=
		\{[(a,s),\chi]:\chi\in D_{N\ast_\Lambda S}((a,s)^*(a,s))\}: (a,s)\in N\ast_{\Lambda}S\big\},
		\]
		and by Proposition~\ref{prop:twistedgroupoid} a basis of compact open bisections for
		$\GrpG(N)\ast_{\hat\Lambda}\GrpG(S)$ is the set
		\[
		\big\{\mathcal{D}_{a,s}
		:=
		\{([a,\hat\alpha(\beta_s(\eta))],[s,\eta]):\eta\in D_S(s^*s)\}:(a,s)\in N\ast_\Lambda S\big\}.
		\]
		We see from the formula \eqref{eq:defofpsi} that $\psi$ maps $D_{N\ast_{\Lambda}S}(a,s)$ onto $\mathcal{D}_{a,s}$ as long as we can prove that
		\[
		\chi\in D_{N\ast_\Lambda S}((a,s)^*(a,s)) \implies \widetilde{\chi}\in D_S(s^*s).
		\]
		But this follows as $(a,s)^*(a,s)=(\alpha(s^*s),s^*s)$, and so we have $\widetilde{\chi}(s^*s)=\chi(\alpha(s^*s),s^*s)=1$. Hence $\psi$ maps each $D_{N\ast_{\Lambda}S}(a,s)$ bijectively onto each $\mathcal{D}_{a,s}$, and so $\psi$ is a homeomorphism.
		
		Now consider the composition $\psi\circ \GrpG(\varphi)$. We first claim that for each $\chi\in \widehat{E}(T)$ we have
		\[
		\widetilde{(\chi\circ(\varphi|_{E(T)})^{-1})}=\chi\circ j|_{E(S)}.
		\]
		To see this, let $e\in E(S)$. Then
		\[
		\widetilde{\chi\circ(\varphi|_{E(T)})^{-1}}(e)= \chi\circ (\varphi|_{E(T)})^{-1}(\alpha(e),e)=\chi(j(e))=\chi\circ j|_{E(S)}(e),
		\]
		and hence $\widetilde{\chi\circ(\varphi|_{E(T)})^{-1}}=\chi\circ j|_{E(S)}$. Then we have
		\begin{align*}
			\psi\circ\GrpG(\varphi)([t,\chi]) &=\Big([a_t,\widehat{\alpha}(\beta_{\pi(t)}(\widetilde{\chi\circ(\varphi|_{E(T)})^{-1}}))] , [\pi(t),\widetilde{\chi\circ(\varphi|_{E(T)})^{-1}}] \Big)\\
			&= \bigl([a_t,\widehat{\alpha}(\beta_{\pi(t)}(\chi\circ j|_{E(S)}))],[\pi(t),\chi\circ j|_{E(S)}]\bigr).
		\end{align*}
		So defining $\widehat{\varphi}:=\psi\circ\GrpG(\varphi)$ gives us our desired isomorphism of topological groupoids.
	\end{proof}
	
	\begin{remark}\label{rem:withoutorderpreserving}
		We emphasise here that we have used the order-preserving assumption in the proofs of Proposition~\ref{prop:dualtriple} (iii) and (iv). In particular, the functions $\widehat{f}$ and $[s,\chi]\mapsto \widehat{\lambda}_{[s,\chi]}$ require $j$ to be order preserving to be well defined. So we can only induce a groupoid twisted action $\widehat{\Lambda}$ from a twisted $S$-module $\Lambda$ when $j$ is order preserving.
	\end{remark}
	
	We can now state and prove this result at the level of an arbitrary extension of ample groupoids that admits a continuous global section.
	
	\begin{theorem}[Theorem A]\label{thm:globalcrossedproduct}
		Let
		\[
		\xymatrix{\GrpH\ar[r]^{\kappa}  & \Sigma \ar[r]^{\rho} & \GrpG} ,
		\]
		be an extension of ample groupoids, with $\gamma\colon \GrpG\to \Sigma$ a continuous global section of $\rho$. Then there exists a groupoid twisted action $\widehat{\Lambda}:=(\widehat{\alpha},\widehat{\lambda},\widehat{f})$ consisting of a homeomorphism $\widehat{\alpha}\colon \GrpG^{(0)}\to \GrpH^{(0)}$; a continuous $\widehat{f}$-twisted action $\widehat{\lambda}\colon \GrpG\to \operatorname{End}(\GrpH)$ where each $\widehat{\lambda}_g$ is a bijection $\GrpH_{\widehat{\alpha}(d_\GrpG(g))}\to \GrpH_{\widehat{\alpha}(r_\GrpG(g))}$; and a continuous $2$-cocycle $\widehat{f}\colon \GrpG^{(2)}\to \GrpH$ for $\widehat{\lambda}$. Moreover, the set
		\[
		\GrpH\ast_{\widehat{\Lambda}}\GrpG:=\{(h,g)\in \GrpH\times \GrpG : d_\GrpH(h)=\widehat{\alpha}(r_\GrpG(g))\}
		\]
		can be endowed with a topological groupoid structure, including the product
		\[
		(h_1,g_1)(h_2,g_2)
		=
		\bigl(h_1\,\widehat{\lambda}_{g_1}(h_2)\,\widehat{f}(g_1,g_2),\,g_1g_2\bigr)\text{ for }(g_1,g_2)\in \GrpG^{(2)},
		\]
		such that $\GrpH\ast_{\widehat{\Lambda}}\GrpG$ is isomorphic to $\Sigma$.
	\end{theorem}
	
	\begin{proof}
		Let $N:=\Gamma_c(\GrpH)$, $T:=\Gamma_c(\Sigma)$, and $S:=\Gamma_c(\GrpG)$ be the inverse semigroups of compact open bisections. The functors
		$\kappa:\GrpH\to \Sigma$ and $\rho:\Sigma\to \GrpG$ induce inverse-semigroup
		homomorphisms $\Gamma_c(\kappa):N\to T$ and $\Gamma_c(\rho)\colon T\to S$. Since $\gamma\colon \GrpG\to\Sigma$ is a continuous section of $\rho$, defining $j(U):=\gamma(U)$ for $U\in \Gamma_c(\GrpG)$ gives an order-preserving section $j\colon S\to T$	of $\Gamma_c(\rho)$. Hence
		\[
		\xymatrix{N\ar[r]^{\iota}  & T \ar[r]^{\pi} & S\ar@/^8pt/ [l] ^{j}}
		\]
		is an extension of inverse semigroups with order-preserving section $j$. Our previous results in this section (Proposition~\ref{prop:dualtriple}, Proposition~\ref{prop:twistedgroupoid}, and Theorem~\ref{thm:mainthmwithorderpreserving}) now give the result.
	\end{proof}
	
	We know from \cite[Proposition~3.9]{S2} that when $\GrpG(S)$ is Hausdorff and $\GrpG(S)\setminus \GrpG(S)^{(0)}$ is paracompact (which happens, for example, when both $\GrpG(S)$ and $\GrpG(T)$ are second countable and Hausdorff \cite[Lemma~2.4]{A9}), then there exists a continuous global section $\gamma: \GrpG(S)\rightarrow \GrpG(T)$ with $\gamma(\GrpG(S)^{(0)})\subseteq \GrpG(T)^{(0)}$. So we can use Theorem~\ref{thm:globalcrossedproduct} to obtain the following.
	
	\begin{corollary}\label{Cor:ClassicCase}
		Let $\xymatrix{\mathcal{H} \ar[r]^{\kappa}  & \Sigma \ar[r]^{\rho} & \GrpG}$ be an extension of ample groupoids with $\kappa, \rho$ open iso-unital functors. If $\GrpG$ is Hausdorff and $\GrpG \setminus \GrpG^{(0)}$ is paracompact, then $\Sigma$ is isomorphic to $\GrpH\ast_{\widehat{\Lambda}}\GrpG$ for some groupoid twisted action $\widehat{\Lambda}$.
	\end{corollary}
	
	We close this section with two examples that give a taste of the potential applications.
	
	\begin{example}\label{Exam:OpenIsotropy}
		Let $\GrpG$ be a second countable, ample, Hausdorff groupoid. Then $\text{Iso}(\GrpG)^\circ$ is an open, normal subgroupoid containing $\GrpG^{(0)}$. Hence if $\text{Iso}(\GrpG)^\circ$ is also closed, \cite[Proposition~2.5(d)]{SW16} says that $\GrpG / \text{Iso}(\GrpG)^\circ$ is a second countable, ample, Hausdorff groupoid, and $(\text{Iso}(\GrpG)^\circ)^{(0)}\cong \GrpG^{(0)}\cong (\GrpG / \text{Iso}(\GrpG)^\circ)^{(0)}$. Thus
		\[
		\xymatrix{
			\text{Iso}(\GrpG)^\circ \ar@{^{(}->}[r] & \GrpG \ar@{->>}[r] & \GrpG / \text{Iso}(\GrpG)^\circ
		}
		\]
		is an extension of ample groupoids, and so by Corollary \ref{Cor:ClassicCase} we have
		\[
		\GrpG\cong \text{Iso}(\GrpG)^\circ\ast_{\widehat{\Lambda}}(\GrpG / \text{Iso}(\GrpG)^\circ)
		\]
		for a suitable groupoid twisted action $\widehat{\Lambda}$.
		
		If we write the dual extension of Boolean inverse semigroups
		$$\xymatrix{N\ar[r]^{\iota}  & T \ar[r]^{\pi} & S},$$
		then $N=Z(E(T))$ by \cite[Proposition 2.9]{SZ}, and $S\cong T/Z(E(T))$ is a fundamental inverse semigroup, so that $\GrpG / \text{Iso}(\GrpG)^\circ\cong \GrpG(S)$ is effective \cite[Proposition 2.10]{SZ}. This example means that any second countable, ample, Hausdorff groupoid in which
		the interior of the isotropy is closed can be written as a twisted crossed product
		of a group bundle by an effective ample groupoid, extending to the context of \emph{Hausdorff} groupoids the classical result of inverse semigroups \cite[Chapter 5]{Law}.
	\end{example}
	
	\begin{example}\label{Exam:ClosureUnitSpace}
		Let $\GrpG$ be a second countable, ample groupoid such that $\overline{\GrpG^{(0)}}$ is open in $\GrpG$ (for example, this is automatic if $\GrpG$ is Hausdorff, or if $\GrpG^{(0)}$ is \emph{extremally disconnected} in the sense that the closure of every open subset of $\GrpG^{(0)}$ is open). Since the range and source maps are continuous, $\overline{\GrpG^{(0)}}\subseteq \text{Iso}(\GrpG)^\circ$. Following the notation in \cite[Section 5]{KKLRU}, $\GrpG_{\text{Haus}}:=\GrpG / \overline{\GrpG^{(0)}}$ is a second countable, ample, Hausdorff groupoid, and $\overline{\GrpG^{(0)}}^{(0)}\cong \GrpG^{(0)}\cong \GrpG_{\text{Haus}}^{(0)}$. Thus,
		\[
		\xymatrix{
			\overline{\GrpG^{(0)}} \ar@{^{(}->}[r] & \GrpG \ar@{->>}[r] & \GrpG_{\text{Haus}}
		}
		\]
		is an extension of ample groupoids, and so by Corollary \ref{Cor:ClassicCase} we have $\GrpG\cong \overline{\GrpG^{(0)}}\ast_{\widehat{\Lambda}}\GrpG_{\text{Haus}}$ for a suitable groupoid twisted action $\widehat{\Lambda}$. That is, in some sense, the failure of Hausdorffness is isolated in the non-Hausdorff group bundle $\overline{\GrpG^{(0)}}$.
	\end{example}

	\section{Examples with no continuous global section}\label{Sect:Examples}
	
	In Section \ref{sec:withorderpreserving} we proved that whenever we have an extension of inverse semigroups as in \eqref{eq:extensionINVSMGP} with an order-preserving section $j:S\rightarrow T$, the groupoid $\GrpG(T)$ is ``globally trivial'', in the sense that it can be written as a twisted crossed product $\GrpG(N)\ast_{\widehat{\Lambda}}\GrpG(S)$. But without an order-preserving section $j:S\rightarrow T$, as we discussed in Remark~\ref{rem:withoutorderpreserving}, there is no groupoid twisted action $\widehat{\Lambda}$, and hence no twisted crossed product groupoid. In this section, we describe two examples of ample groupoid extensions that do not have continuous global sections. The extensions we consider have the form
	\[
	\xymatrix{\GrpH\ar[r]^{\kappa}  & \Sigma \ar[r]^{\rho} & \GrpG },
	\]
	with $\kappa, \rho$ open continuous homomorphisms such that $\kappa_{\vert {\GrpH^{(0)}}}$ and $\rho_{\vert {\Sigma^{(0)}}}$ are iso-unital, with $\kappa$ injective, $\rho$ surjective, and $\kappa(\GrpH)=\rho^{-1}(\GrpG^{(0)})$. In our examples, $\GrpH$ is a group bundle.
	
	\begin{example}\label{Exa:Example1}
		Fix a totally disconnected, locally compact, Hausdorff space $X$.  Fix $p\in X$ and write the $2$-element group
		$\mathbb{Z}_2=\{0,1\}.$
		Define
		\[
		\GrpH= (X\times \ZZ_2)\setminus \{(p, 1)\},\quad \Sigma= X\times \ZZ_2,\quad\text{and}\quad \GrpG=X\sqcup \{(p, 1)\}
		\]
		where $\GrpH^{(0)}=\Sigma^{(0)}=X\times \{0\}$ and $\GrpG^{(0)}=X$. Both $\GrpH$ and $\Sigma$ are group bundles with the relative product topologies and hence are ample Hausdorff groupoids.  Notice that $\GrpH$ is not of the form $\GrpH^{(0)}\times T$ for any abelian group $T$.
		The groupoid $\GrpG$ is the ``two-headed snake'', which is the trivial groupoid $X$ along with a single point of nontrivial isotropy at $(p,1)$ where $s((p, 1))=r((p,1)) = p$. The topology on $\GrpG$ is the topology on $X$ and for each open $U \subseteq X$ such that $p \in U$ we add an additional open set
		\[
		(U\setminus\{p\}) \cup \{(p, 1)\}.
		\]
		Thus $\GrpG$ is an ample groupoid that is not Hausdorff.
		Let $\kappa:\GrpH \to \Sigma$ be inclusion and $\rho:\Sigma \to \GrpG$ be such that
		\[
		\rho(x, i) =
		\begin{cases}
			(p, 1) & \text{if } (x,i)=(p, 1)\\
			x&\text{otherwise.}
		\end{cases}
		\]
		Then it is straightforward to check that
		\[
	\xymatrix{\GrpH\ar[r]^{\kappa}  & \Sigma \ar[r]^{\rho} & \GrpG }
		\]
		is a groupoid extension. We claim that no continuous global section exists. To see this, notice that any sequence $\{x_n\}\subseteq \GrpG$ converging to $p\in \GrpG$ also converges to $(p, 1)$. Since $\Sigma$ is Hausdorff, limits of sequences are unique and hence any continuous function $\psi: \GrpG\rightarrow \Sigma$ must send both $p$ and $(p,1)$ to the same point of $\Sigma$.  But a section of $\rho$ must send $p$ to $(p,0)$ and $(p,1)$ to $(p, 1)$, thus no continuous section exists.
		
		Let us check what happens with the section of the dual inverse semigroup extension
		\[
		\xymatrix{\Gamma_c(\GrpH)\ar[r]^{\Gamma_c(\kappa)}  & \Gamma_c(\Sigma) \ar[r]^{\Gamma_c(\rho)} & \Gamma_c(\GrpG)}.
		\]
		We first determine the partial order in $\Gamma_c(\Sigma)$ and  $\Gamma_c(\GrpG)$. Since all elements $\gamma$ in both group bundles $\GrpG$ and $\Sigma$ are such that $\gamma=\gamma^{-1}$, it is straightforward to check that for compact open bisections $U$ and $V$ that are both in either $\Gamma_c(\GrpG)$ or  $\Gamma_c(\Sigma)$, we have $U\leq V$ if and only if $U\subseteq V$.
		
		We claim that there is no order-preserving global section for $\Gamma_c(\rho)$. To see this, suppose for contradiction that $j:\Gamma_c(\GrpG)\rightarrow \Gamma_c(\Sigma)$ is an order-preserving global section for $\Gamma_c(\rho)$. Fix $V \in \Gamma_c(\GrpG)$ such that $(p,1) \in V$. Let $B = j(V) \cap (X \times \{1\})$. Since $\Sigma$ is Hausdorff and $j$ is a section,  $B$ is a compact open bisection containing $(p,1)\in \Sigma$.  We then have $\Gamma_c(\rho)(B) \subseteq V$.  Let $U$ be a compact open bisection contained in $\Gamma_c(\rho)(B) \cap X$.
		Then $U \leq \Gamma_c(\rho)(B)$.  But $U$ is an idempotent, so $j(U)$ must also be an idempotent.  However, since $j$ is order preserving, we must have $j(U) \subseteq B$ and hence $j(U)$ cannot be an idempotent, which is a contradiction. So the claim holds and there are no order-preserving global sections of $\Gamma_c(\rho)$, as predicted by \cite[Proposition 3.8]{S2} in view of Example \ref{Exa:Example1}.
	\end{example}
	
	\begin{remark}
		Even though there is no continuous global section for the extension in Example~\ref{Exa:Example1}, we can construct two ``local'' sections of $\rho$ by taking:
		\begin{enumerate}
			\item the identity map between compact open subsets of the unit spaces, and
			
			\item the map that takes any compact open subset of $\GrpG\setminus \{p\}$ to a compact open subset of $X \times \{1\}$
			in the obvious way.
		\end{enumerate}
		Both local sections are well-defined, and their respective maps on the level of inverse semigroups are order-preserving sections of $\Gamma_c(\rho)$ on their respective domains. This idea of moving locally will motivate our work in Section~\ref{sec:nonorderpreserving} where we don't assume the existence of a continuous global section.
	\end{remark}
	
	We now show that the same pathology can happen for extensions such that $\Sigma$ is a non-Hausdorff groupoid.
	
	\begin{example}\label{Exa:Example2}
		Fix a second countable, totally disconnected, locally compact, Hausdorff space $X$,
		and fix $p\in X$. We construct an extension
		\[
		\xymatrix{\GrpH\ar[r]^{\kappa}  & \Sigma \ar[r]^{\rho} & \GrpG },
		\]
		in which $\Sigma$ is non-Hausdorff and no continuous global section of $\rho$
		exists.
		
		As a set, let
		\[
		\Sigma
		:=
		(X\times\{0\})\sqcup (X\times\{1\})
		\sqcup \{(p,2),(p,3)\}.
		\]
		The unit space is $\Sigma^{(0)}=X\times\{0\}$. We make $\Sigma$ into a group bundle as follows. For $x\neq p$, the fibre
		$x\Sigma x$ is the two-element group $\{(x,0),(x,1)\}\cong \mathbb Z_2$; at $p$, the fibre $p\Sigma p$ is the four-element
		group $\{(p,0),(p,1),(p,2),(p,3)\}\cong \mathbb Z_4$. To define a topology on $\Sigma$, for each compact open subset $U\in\operatorname{CO}(X)$,
		put
		\[
		\begin{aligned}
			U_0 &:= U\times\{0\},\\
			U_1 &:= U\times\{1\},\\
			U_2 &:= ((U\setminus\{p\})\times\{0\})\cup\{(p,2)\},\\
			U_3 &:= ((U\setminus\{p\})\times\{1\})\cup\{(p,3)\}.
		\end{aligned}
		\]
		We give $\Sigma$ the topology with basis $\mathcal B_\Sigma:= \{U_i: U\in \operatorname{CO}(X),\ 0\leq i\leq 3\}$. With this topology, $\Sigma$ is an ample group bundle. It is non-Hausdorff because $(p,0)$ and $(p,2)$ cannot be separated by disjoint open
		neighbourhoods (and similarly, $(p,1)$ and $(p,3)$ cannot be separated).
		
		Next define $\mathcal{H}$ to be the subgroup bundle of $\Sigma$ given by
		\[
		\mathcal{H}
		:=
		(X\times\{0\})\sqcup ((X\setminus\{p\})\times\{1\}).
		\]
		Then for $x\not= p$ we have $x\mathcal{H}x\cong \mathbb Z_2$, and $p\mathcal{H}p\cong \{0\}$. We give $\mathcal{H}$ the subspace topology inherited from $\Sigma$. Equivalently, $\mathcal{H}$
		has basis
		\[
		\{U_0 : U\in\operatorname{CO}(X)\}
		\cup
		\{U_1 : U\in\operatorname{CO}(X),\ p\notin U\}.
		\]
		Hence $\mathcal{H}$ is a Hausdorff ample abelian group bundle. Let $\kappa:\mathcal{H}\hookrightarrow \Sigma$ be the inclusion map. Then $\kappa$ is an injective, continuous, open groupoid
		homomorphism, and $\kappa|_{\mathcal{H}^{(0)}}:\mathcal{H}^{(0)}\to \Sigma^{(0)}$ is a homeomorphism.
		
		As a set, let
		\[
		\mathcal{G}
		:=
		X\sqcup\{(p,1),(p,2),(p,3)\},
		\]
		where we identify the unit space $\mathcal{G}^{(0)}$ with $X$. We make $\mathcal{G}$ into a group
		bundle by declaring the fibre over $x\neq p$ to be trivial, and the fibre over
		$p$ to be
		\[
		p\mathcal{G}p
		=
		\{p,(p,1),(p,2),(p,3)\}
		\cong \mathbb Z_4.
		\]
		For each compact open subset $U\in\operatorname{CO}(X)$ define
		\[
		\begin{aligned}
			\widehat U_0 &:= U,\\
			\widehat U_1 &:= (U\setminus\{p\})\cup\{(p,1)\},\\
			\widehat U_2 &:= (U\setminus\{p\})\cup\{(p,2)\},\\
			\widehat U_3 &:= (U\setminus\{p\})\cup\{(p,3)\}.
		\end{aligned}
		\]
		We give $\mathcal{G}$ the topology with basis $\mathcal B_\mathcal{G}:=\{\widehat U_i: U\in\operatorname{CO}(X),\ 0\leq i\leq 3\}$. Then $\mathcal{G}$ is an ample group bundle. It is non-Hausdorff, since $p$ and
		$(p,i)$ cannot be separated by disjoint open neighbourhoods for $i=1,2,3$.
		
		Define $\rho:\Sigma\to \mathcal{G}$ by
		\[
		\rho(x,i)
		=
		\begin{cases}
			(p,i) & \text{if }x=p,\\
			x      & \text{if }x\neq p.
		\end{cases}
		\]
		Then $\rho$ is an open, continuous, surjective groupoid homomorphism such that $\kappa(\GrpH)=\rho^{-1}(\GrpG^{(0)})$, and
		\[
		\xymatrix{\GrpH\ar[r]^{\kappa}  & \Sigma \ar[r]^{\rho} & \GrpG }
		\]
		is a groupoid extension. A similar argument to the one presented in Example~\ref{Exa:Example1} shows that there is no continuous global section for $\rho$.
		
	\end{example}

	\section{Groupoid extensions with no continuous global section}\label{sec:nonorderpreserving}
	
	Throughout this section we assume that \[
	\xymatrix{N\ar[r]^{\iota}  & T \ar[r]^{\pi} & S\ar@/^8pt/ [l] ^{j}}
	\]
	is an inverse semigroup extension with section $j:S\rightarrow T$, but we no longer assume $j$ is order-preserving. Without this assumption we no longer have a well-defined global section $\widehat{j}\colon \mathcal{G}(S)\to\mathcal{G}(T)$ at the level of groupoids, and we no longer have a groupoid twisted action $\widehat{\Lambda}=(\widehat{\alpha},\widehat{\lambda},\widehat{f})$, meaning we have no description of $\mathcal{G}(T)$ as a twisted crossed product groupoid.
	
	In place of this description we instead construct a space $\mathcal{X}=\bigcup_{t\in T}\mathcal{X}_t$, a family of locally-defined maps from $\mathcal{G}(T)$ to $\mathcal{X}$, and a locally-defined product on $\mathcal{X}$ that is compatible with these local maps. Even though we no longer have a twisted crossed product groupoid $\GrpG(N)\ast_{\widehat{\Lambda}}\GrpG(S)$, we do still have the set of elements of $\GrpG(N)\ast_{\widehat{\Lambda}}\GrpG(S)$ (as $\widehat{\Lambda}$, when it exists, only gets involved in the algebraic properties of the groupoid), and our space $\mathcal{X}$ will be in bijection with this set of elements. Each local space $\mathcal{X}_t$ will be in bijection with the set of elements in the compact open bisections $\mathcal{D}_{a_t,\pi(t)}$ from Proposition~\ref{prop:twistedgroupoid}.
	
	To construct our locally-defined maps from $\GrpG(T)$ to $\mathcal{X}$, we start by defining a family of local maps from $T$ to $N\ast_\Lambda S$, and a local product on $N\ast_\Lambda S$, which will induce their topological analogues.
	
	\subsection{Algebraic local charts and the algebraic local product}\label{subsec:alglocalcharts}
	
	Given an inverse semigroup $S$ and an element $s\in S$ we denote by $s\!\downarrow$ the set
	\[
	{s\!\downarrow} :=\{ r\in S : r\leq s\}
	\]
	Note that we can also describe $s\!\downarrow$ as the set $\{se : e\in E(S), e\leq s^*s\}$.
	
	\begin{notation}
		For each $s\in S$ we denote
		\[
		R_\alpha(s) :=  \{ (\alpha(pp^*),p) : p \in s\!\downarrow\} \subseteq N\ast_\Lambda S.
		\]
	\end{notation}
	
	Since $\alpha\circ \pi|_{E(T)}=\iota^{-1}$, for each $t\in T$ we have
	\[
	R_\alpha(\pi(t)) = \{(\iota^{-1}(rr^*),\pi(r)) : r\in t\!\downarrow\}.
	\]
	
	\begin{proposition}\label{prop:LocalSections}
		For each $t\in T$ there is an order-preserving (set-theoretic) bijection $\varphi_t\colon t\!\downarrow\to R_\alpha(\pi(t))$ given by
		\[
		\varphi_t(r)=(\iota^{-1}(rr^*),\pi(r)),
		\]
		and with inverse given by $\varphi_t^{-1}(\alpha(pp^*),p)= tj(p^*p)$.
	\end{proposition}
	
	\begin{proof}
		Write $s:=\pi(t)$. For $r\in t\!\downarrow$ we have $r=tr^*r\implies \pi(r)=s\pi(r^*r)$, and $r^*r\le t^*t\implies \pi(r^*r)\le \pi(t^*t)=s^*s$. This shows that $\pi(r)\in s\!\downarrow$, and hence $(\alpha(\pi(rr^*)),\pi(r))\in R_\alpha(s)$. We define $\varphi_t\colon t\!\downarrow\to R_\alpha(s)$ by $\varphi_t(r)=(\alpha(\pi(rr^*)),\pi(r))$.
		
		We now show that $\varphi_t$ has an inverse. First note that for $p\in s\!\downarrow$ we have
		\[
		p^*p\le s^*s\implies j(p^*p)=\pi^{-1}(p^*p)\le \pi^{-1}(s^*s)=\pi^{-1}(\pi(t^*t))=t^*t.
		\]
		Hence $tj(p^*p)\in t\!\downarrow$, and we define $\rho: R_\alpha(s)\to t\!\downarrow$ by $\rho(\alpha(pp^*),p)= tj(p^*p)$. For each $(\alpha(pp^*),p)\in R_\alpha(s)$ we have
		\begin{align*}
			\varphi_t(\rho(\alpha(pp^*),p)) &= \varphi_t(tj(p^*p)) \\
			&= (\alpha(\pi(tj(p^*p)(tj(p^*p))^*)),\pi(tj(p^*p)))\\
			&= ( \alpha(sp^*p(sp^*p)^*),sp^*p)\\
			&= (\alpha(pp^*),p),
		\end{align*}
		and for each $r\in t\!\downarrow$ we have
		\[
		\rho(\varphi_t(r)) = \rho( \alpha(\pi(rr^*)),\pi(r))= tj(\pi(r)^*\pi(r))= tr^*r=r.
		\]
		Hence $\varphi_t$ is bijective with inverse $\rho$.
		
		To see that $\varphi_t$ is order preserving, let $r_1, r_2\in T$ with $r_1\le r_2$. Then $\pi(r_1)\le \pi(r_2)$, and we have
		\[
		\alpha(\pi(r_1r_1^*))=\alpha(\pi(r_1r_1^*r_2r_2^*))=\alpha(\pi(r_1)\pi(r_1)^*)\alpha(\pi(r_2r_2^*)).
		\]
		We see from \eqref{eq:orderondecomp} that $\varphi_t(r_1)=(\alpha(\pi(r_1r_1^*)),\pi(r_1))\le (\alpha(\pi(r_2r_2^*)),\pi(r_2))=\varphi_t(r_2)$.
	\end{proof}
	
	\begin{definition}\label{def:alglocalcharts}
		We call each map $\varphi_t\colon t\!\downarrow\to R_\alpha(\pi(t))$ from Proposition~\ref{prop:LocalSections} an {\em algebraic local chart}, and the set $\{\varphi_t : t\in T\}$ the {\em family of algebraic local charts}.
	\end{definition}
	
	\begin{definition}\label{def:alglocalproduct}
		Let $t_1,t_2\in T$ and $r_1\in t_1\!\downarrow$, $r_2\in t_2\!\downarrow$. We define the {\em algebraic local product} of $(\iota^{-1}(r_1r_1^*),\pi(r_1))\in R_\alpha(\pi(t_1))$ and $(\iota^{-1}(r_2r_2^*),\pi(r_2))\in R_\alpha(\pi(t_2))$ to be
		\[
		(\iota^{-1}(r_1r_1^*),\pi(r_1)) \cdot^{\operatorname{alg}}_{t_1,t_2} (\iota^{-1}(r_2r_2^*),\pi(r_2)) := (\iota^{-1}(r_1r_2(r_1r_2)^*),\pi(r_1r_2))\in R_{\alpha}(\pi(t_1t_2)).
		\]	
		Note that this means for $t_1,t_2\in T$ and $r_1\in t_1\!\downarrow$, $r_2\in t_2\!\downarrow$ we have
		\[
		\varphi_{t_1}(r_1) \cdot^{\operatorname{alg}}_{t_1,t_2}\varphi_{t_2}(r_2)=\varphi_{t_1t_2}(r_1r_2).
		\]
	\end{definition}
	
	\begin{remark}\label{Rem:GoodDefAlgProd}
		Notice that $\varphi_{t_1}(r_1) \cdot^{\operatorname{alg}}_{t_1,t_2}\varphi_{t_2}(r_2)\ne \varphi_{t_1}(r_1) \cdot\varphi_{t_2}(r_2)\in N\ast_\Lambda S$, but it is the right definition, as we will see in Remark \ref{RemConnection}.
	\end{remark}
	
	\subsection{Topological local charts and the topological local product}\label{subsec:toplocalcharts}
	
	Our main result in this section is Theorem~\ref{thm:mainthmnoorderpreserving}, which is a local analogue of Theorem~\ref{thm:mainthmwithorderpreserving} in the setting of general extensions of inverse semigroups (without the assumption of an order-preserving section). We start by defining the space $\mathcal{X}$.
	
	\begin{definition}\label{def:XtandX}
		For each $t\in T$ we define
		\[
		\mathcal{X}_t:=\left\{
		\left([a_t,\widehat{\alpha}(\beta_{\pi(t)}(\chi))],[\pi(t),\chi]\right)
		: \chi\in D_S(\pi(t^*t))
		\right\}\subseteq \GrpG(N)\times\GrpG(S),
		\]
		and we define
		\[
		\mathcal{X}=\bigcup_{t\in T}\mathcal{X}_t.
		\]
		Note that, as a set, $\mathcal{X}_t$ is the compact open bisection $\mathcal{D}_{a_t,\pi(t)}$ as given in Proposition~\ref{prop:twistedgroupoid}. Since the map $\chi\mapsto	\left([a_t,\widehat{\alpha}(\beta_{\pi(t)}(\chi))],[\pi(t),\chi]\right)$ is a bijection, we give $\mathcal{X}_t$ the unique topology that makes this map a homeomorphism. We then give $\mathcal X$ the final topology with respect to the inclusions $\mathcal{X}_t\hookrightarrow \mathcal{X}$.
	\end{definition}
	
	In the next lemma we use our algebraic local charts to induce local maps $\GrpG(T)\to \GrpG(N\ast_\Lambda S)$. We will then prove the existence of local maps from $\GrpG(N\ast_\Lambda S)$ to $\mathcal{X}$, and we will compose these maps to get our topological local charts $\Phi_t$ in Definition~\ref{def:toplocalcharts}.
	
	\begin{lemma}
		For each $t\in T$ the map $\widehat{\varphi_t}\colon D_T(t)\to D_{N\ast_\Lambda S}(\varphi_t(t))$ given by
		\[
		\widehat{\varphi_t}([t,\chi]) = [\varphi_t(t),\chi\circ(\varphi_t|_{E(T)\cap t\!\downarrow})^{-1}]
		\]
		is a homeomorphism.
	\end{lemma}

	\begin{proof}
		We first show that $\widehat{\varphi}_t$ is well defined. First, since $\varphi_t$ is order-preserving on $t\!\downarrow$ by Proposition \ref{prop:LocalSections}, the definition of $\widehat{\varphi_t}$ does not depend on the choice of the representative of $\varphi_t(t)$. Now, let $[t,\chi]\in D_T(t)$, so
		$\chi(t^*t)=1$. Put
		\[
		\eta:=\chi\circ(\varphi_t|_{E(T)\cap t\!\downarrow})^{-1}.
		\]
		Then $\eta$ is a character on $E(N\ast_\Lambda S)\cap \varphi_t(t)\!\downarrow$. Since $\varphi_t(t)^*\varphi_t(t)=(\iota^{-1}(t^*t),\pi(t^*t))=\varphi_t(t^*t)$, we have
		\[
		\eta(\varphi_t(t)^*\varphi_t(t))
		=
		\eta(\varphi_t(t^*t))
		=
		\chi(t^*t)
		=
		1.
		\]
		Thus
		$[\varphi_t(t),\eta]\in D_{N\ast_\Lambda S}(\varphi_t(t))$, and so
		$\widehat{\varphi}_t([t,\chi])$ is well defined.
		
		To see that $\widehat{\varphi}_t$ is injective, suppose that $\widehat{\varphi}_t([t,\chi_1])=\widehat{\varphi}_t([t,\chi_2])$. Then
		\[
		[\varphi_t(t),\chi_1\circ(\varphi_t|_{E(T)\cap t\!\downarrow})^{-1}]
		=
		[\varphi_t(t),\chi_2\circ(\varphi_t|_{E(T)\cap t\!\downarrow})^{-1}].
		\]
		Hence
		\[
		\chi_1\circ(\varphi_t|_{E(T)\cap t\!\downarrow})^{-1}
		=
		\chi_2\circ(\varphi_t|_{E(T)\cap t\!\downarrow})^{-1},
		\]
		and so $\chi_1=\chi_2$.
		
		For surjectivity, let $[\varphi_t(t),\eta]\in D_{N\ast_\Lambda S}(\varphi_t(t))$. Define
		\[
		\chi:=\eta\circ \varphi_t|_{E(T)\cap t\!\downarrow}.
		\]
		Then $\chi$ is a character on $E(T)\cap t\!\downarrow$, and
		\[
		\chi(t^*t)
		=
		\eta(\varphi_t(t^*t))
		=
		\eta(\varphi_t(t)^*\varphi_t(t))
		=
		1.
		\]
		Thus $[t,\chi]\in D_T(t)$, and by construction $\widehat{\varphi}_t([t,\chi])
		=
		[\varphi_t(t),\eta]$. So $\widehat{\varphi}_t$ is surjective, and hence is bijective.
		
		Finally, $\widehat{\varphi}_t$ is a homeomorphism because it is induced by the semilattice
		isomorphism
		\[
		\varphi_t|_{E(T)\cap t\!\downarrow}:E(T)\cap t\!\downarrow \to E(N\ast_\Lambda S)\cap \varphi_t(t)\!\downarrow.
		\]
	\end{proof}
	
	Recall that for $\chi\in \widehat{E}(N\ast_\Lambda S)$ we denote by $\widetilde{\chi}$ the character in $\widehat{E}(S)$ given by $\widetilde{\chi}(e)=\chi(\alpha(e),e)$.
	
	\begin{lemma}\label{lem: the Psi maps}
		For each $t\in T$ there is a homeomorphism $\Psi_t: D_{N\ast_\Lambda S}(\varphi_t(t))\to \mathcal{X}_{t}$ given by
		\[
		\Psi_t([\varphi_t(t),\chi])=\big([a_t,\widehat{\alpha}(\beta_{\pi(t)}(\widetilde{\chi}))],[\pi(t),\widetilde{\chi}]\big),
		\]
		for $\chi\in D_{N\ast_\Lambda S}(\varphi_t(t)^*\varphi_t(t))$.
	\end{lemma}
	
	\begin{proof}
		First, notice that the action of $\Psi_t$ on $t\!\downarrow\!$ is exactly the composition $\varphi\circ \varphi_t^{-1}$, and since both maps are order-preserving on their domains, then so is $\Psi_t$; hence, $\Psi_t$ is well-defined.
		
		In fact, we can see $\Psi_t$ as the restriction of $\psi$ to the bisection $D_{N\ast_\Lambda S}(\varphi_t(t))$ from the proof of Theorem~\ref{thm:mainthmwithorderpreserving}. We saw that $\psi$ maps $D_{N\ast_\Lambda S}(\varphi_t(t))$ onto $\mathcal{D}_{a_t,\pi(t)}$, which is the same as $\mathcal{X}_t$, so $\Psi_t$ is a surjective map onto $\mathcal{X}_t$. We also know that $\Psi_t$ is injective. To see that $\Psi_t$ is a homeomorphism, first note that the map $\chi\mapsto \widetilde{\chi}$ is the homeomorphism of spectra induced by the semilattice isomorphism $E(S)\to E(N\ast_{\Lambda}S)$ given by $e\mapsto (\alpha(e),e)$. Under this homeomorphism, the compact open set $D_{N\ast_{\Lambda}S}(\varphi_t(t)^*\varphi_t(t))$ corresponds exactly to $D_S(\pi(t)^*\pi(t))$. By definition, $\mathcal{X}_t$ is given the unique
		topology for which the map $\eta\mapsto
		\big([a_t,\widehat{\alpha}(\beta_{\pi(t)}(\eta))],[\pi(t),\eta]\big)$ from $D_S(\pi(t)^*\pi(t))$ onto $\mathcal{X}_t$ is a homeomorphism. Therefore $\Psi_t$ is the composition of two homeomorphisms, and hence is a homeomorphism.
	\end{proof}
	
	\begin{lemma}\label{lem:precisePhitformula}
		For each $t\in T$ the map $\Phi_t:=\Psi_t\circ\widehat{\varphi_t}\colon D_T(t)\to \mathcal{X}_t$ is a homeomorphism given by
		\[
		\Phi_t([t,\chi]) = \big([a_t,\widehat{\alpha}(\beta_{\pi(t)}(\chi\circ j|_{E(S)}))],[\pi(t),\chi\circ j|_{E(S)}]\big)
		\]
		for $\chi\in D_T(t^*t)$.
	\end{lemma}
	
	\begin{proof}
		The range of $\varphi_t|_{E(T)\cap t\!\downarrow}$ is $\{(\alpha(\pi(e)),\pi(e)) : e\in E(T)\cap t\!\downarrow\}$. So for $e\in E(T)\cap t\!\downarrow$ we have
		\[
		\widetilde{\big(\chi\circ \varphi_t|_{E(T)\cap t\!\downarrow}\big)}(\pi(e))= \chi\circ \varphi_t|_{E(T)\cap t\!\downarrow}(\alpha(\pi(e)),\pi(e))=\chi(e)=(\chi\circ j|_{E(S)})(\pi(e)),
		\]
		and hence we have $\widetilde{\big(\chi\circ \varphi_t|_{E(T)\cap t\!\downarrow}\big)}=\chi\circ j|_{E(S)}$. We now have
		\begin{align*}
			\Phi_t([t,\chi]) &= \Psi_t\circ\widehat{\varphi_t}([t,\chi])\\
			&= \Psi_t([\varphi_t(t),\chi\circ(\varphi_t|_{E(T)\cap t\!\downarrow})^{-1}])\\
			&= \big([a_t,\widehat{\alpha}(\beta_{\pi(t)}(\widetilde{\big(\chi\circ \varphi_t|_{E(T)\cap t\!\downarrow}\big)}))],[\pi(t),\widetilde{\big(\chi\circ \varphi_t|_{E(T)\cap t\!\downarrow}\big)}]\big)\\
			&= \big([a_t,\widehat{\alpha}(\beta_{\pi(t)}(\chi\circ j|_{E(S)}))],[\pi(t),\chi\circ j|_{E(S)}]\big),
		\end{align*}
		and so the result holds.
	\end{proof}

	\begin{definition}\label{def:toplocalcharts}
		For each $t\in T$ we call $\Phi_t$ a {\em topological local chart}, and the set $\{\Phi_t : t\in T\}$ the {\em family of topological local charts}.
	\end{definition}
	
	\begin{definition}\label{def:toplocalproduct}
		Let $t_1,t_2\in T$. We define $\mathcal{X}_{t_1}\times_c \mathcal{X}_{t_2}$ to be the set of pairs $(x_1,x_2)$, where
		\[
		x_1=
		\left(
		\left[a_{t_1},\widehat{\alpha}\!\left(\beta_{\pi(t_1t_2)}(\chi)\right)\right],
		\left[\pi(t_1),\beta_{\pi(t_2)}(\chi)\right]
		\right)
		\text{ and }
		x_2=
		\left(
		\left[a_{t_2},\widehat{\alpha}\!\left(\beta_{\pi(t_2)}(\chi)\right)\right],
		\left[\pi(t_2),\chi\right]
		\right),
		\]
		and $\chi\in D_S(\pi(t_1t_2)^*\pi(t_1t_2))$. Note that, in the setting of Section~\ref{sec:withorderpreserving}, $\mathcal{X}_{t_1}\times_c \mathcal{X}_{t_2}$ is the same set as
		\[
		(\mathcal{D}_{a_{t_1},\pi(t_1)}\times \mathcal{D}_{a_{t_2},\pi(t_2)})\cap(\GrpG(N)\ast_{\widehat{\Lambda}}\GrpG(S))^{(2)}.
		\]
		For $(x_1,x_2)\in \mathcal{X}_{t_1}\times_c \mathcal{X}_{t_2}$ we define the {\em topological local product} to be
		\[
		x_1\cdot^{\operatorname{top}}_{t_1,t_2}x_2 := \left(
		\left[a_{t_1t_2},\widehat{\alpha}\!\left(\beta_{\pi(t_1t_2)}(\chi)\right)\right],
		\left[\pi(t_1t_2),\chi\right]
		\right)\in\mathcal{X}_{t_1t_2}.
		\]
	\end{definition}
	
	The next result says that multiplication in $\mathcal{G}(T)$ becomes the topological local product when expressed using the topological local charts, and it is immediate from Lemma~\ref{lem:precisePhitformula} and the definition of the topological local product.
	
	\begin{proposition}\label{prop:topchartsandlocalprods}
		For each $t_1,t_2\in T$, $[t_1,\beta_{t_2}(\chi)]\in D_T(t_1)$, and $[t_2,\chi]\in D_T(t_2)$, we have
		\[
		\Phi_{t_1}([t_1,\beta_{t_2}(\chi)])\cdot^{\operatorname{top}}_{t_1,t_2}\Phi_{t_2}([t_2,\chi]) = \Phi_{t_1t_2}([t_1,\beta_{t_2}(\chi)][t_2,\chi]) = \Phi_{t_1t_2}([t_1t_2, \chi]).
		\]
	\end{proposition}
	
	\begin{remark}
		The family of local charts $\{\Phi_t : t\in T\}$ gives a local-triviality picture for the extension
		\[
		\xymatrix{\GrpG(N)\ar[r]^{\GrpG(\iota)}  & \GrpG(T)\ar[r]^{\GrpG(\pi)} & \GrpG(S) }.
		\]
		In particular, for each $s\in S$, the preimage $\mathcal{G}(\pi)^{-1}(D_S(s))$ is covered by
		the compact open bisections $\{D_T(t):t\in \pi^{-1}(s)\}$.  On each
		such bisection, the chart $\Phi_t$ is a homeomorphism
		\[
		D_T(t)
		\cong
		\left\{
		\big([a_t,\eta],[\pi(t),\chi]\big)
		\in
		D_N(a_t)\times D_S(\pi(t))
		:
		\eta=\widehat{\alpha}(\beta_{\pi(t)}(\chi))
		\right\}.
		\]
		Under this identification, the projection onto the second coordinate gives $\operatorname{pr}_2\circ \Phi_t=	\mathcal{G}(\pi)|_{D_T(t)}$. So $\mathcal{X}_t$ is the coordinate model for the local piece $D_T(t)$ of the extension over $D_S(\pi(t))$.  Proposition~\ref{prop:topchartsandlocalprods} then says that these local trivialisations respect multiplication, in the sense that multiplication in $\GrpG(T)$ becomes the topological local product on the spaces $\mathcal{X}_t$.
	\end{remark}
	
	We now sum up what we have done in the section with the following result.
	
	\begin{theorem}\label{thm:mainthmnoorderpreserving}
		Suppose we have an extension of inverse semigroups
		\[
		\xymatrix{N\ar[r]^{\iota}  & T \ar[r]^{\pi} & S\ar@/^8pt/ [l] ^{j}}
		\]
		with section $j\colon S\to T$, and let
		$\mathcal{X}=\bigcup_{t\in T}\mathcal{X}_t\subseteq \GrpG(N)\times\GrpG(S)$ be as in Definition~\ref{def:XtandX}. Then
		\begin{enumerate}
			\item for each $t\in T$ the topological local
			chart $\Phi_t\colon D_T(t)\to \mathcal{X}_t$ given by
			\[
			\Phi_t([t,\chi]) = \big([a_t,\widehat{\alpha}(\beta_{\pi(t)}(\chi\circ j|_{E(S)}))],[\pi(t),\chi\circ j|_{E(S)}]\big)
			\qquad\text{for } \chi\in D_T(t^*t)
			\]
			is a homeomorphism;
			\item for each $t\in T$ we have
			$\operatorname{pr}_2\circ\,\Phi_t = \GrpG(\pi)|_{D_T(t)}$, where
			$\operatorname{pr}_2\colon\GrpG(N)\times\GrpG(S)\to\GrpG(S)$ is the
			projection onto the second coordinate; and
			\item for all $t_1,t_2\in T$ and $\chi\in D_T((t_1t_2)^*t_1t_2)$, writing
			$\sigma_1:=[t_1,\beta_{t_2}(\chi)]\in D_T(t_1)$ and
			$\sigma_2:=[t_2,\chi]\in D_T(t_2)$, we have
			$(\sigma_1,\sigma_2)\in\GrpG(T)^{(2)}$,
			$\big(\Phi_{t_1}(\sigma_1),\Phi_{t_2}(\sigma_2)\big)\in \mathcal{X}_{t_1}\times_c\mathcal{X}_{t_2}$,
			and
			\[
			\Phi_{t_1}(\sigma_1) \cdot^{\operatorname{top}}_{t_1,t_2} \Phi_{t_2}(\sigma_2)
			= \Phi_{t_1t_2}(\sigma_1\sigma_2).
			\]
		\end{enumerate}
	\end{theorem}
	
	\begin{remark}\label{rmk:localthmforgroupoids}
		We reformulated Theorem~\ref{thm:mainthmwithorderpreserving} in terms of extensions of ample groupoids in Theorem~\ref{thm:globalcrossedproduct} (Theorem A). We now do the same for Theorem~\ref{thm:mainthmnoorderpreserving}. For an arbitrary extension of ample groupoids
		\[
		\xymatrix{H\ar[r]^{\kappa}  & \Sigma \ar[r]^{\rho} & G}
		\]
		with no assumption on the existence of a continuous global section of
		$\rho$, the collections of compact open bisections give us an extension of Boolean inverse semigroups
		\[
		\xymatrix{\Gamma_c(H)\ar[r]^{\Gamma_c(\kappa)}  & \Gamma_c(\Sigma) \ar[r]^{\Gamma_c(\rho)} & \Gamma_c(G)},
		\]
		and isomorphisms $\GrpG(\Gamma_c(H))\cong H$,
		$\GrpG(\Gamma_c(\Sigma))\cong\Sigma$ and $\GrpG(\Gamma_c(G))\cong G$. Since $\Gamma_c(\rho)$ is
		surjective it admits a section $j\colon \Gamma_c(G)\to \Gamma_c(\Sigma)$. Transporting the conclusions of
		Theorem~\ref{thm:mainthmnoorderpreserving} through the above isomorphisms gives:
		\begin{enumerate}
			\item a space $\mathcal{X}:=\bigcup_{U\in\Gamma_c(\Sigma)}\mathcal{X}_U$ that can be identified set-theoretically with $H\ast_{\widehat{\Lambda}}G$, and where each $\mathcal{X}_U$ can be identified set-theoretically with a compact open bisection of $H\ast_{\widehat{\Lambda}}G$; and
			
			\item homeomorphisms $\Phi_U:U\to \mathcal{X}_U$, and for each $U,V\in\Gamma_c(\Sigma)$ a subset $\mathcal{X}_U\times_c \mathcal{X}_V\subseteq \mathcal{X}_U\times \mathcal{X}_V$ on which the formula of Definition~\ref{def:toplocalproduct} describes the local product $\mathcal{X}_U\times_c \mathcal{X}_V \owns (x_1, x_2) \mapsto x_1 \cdot^{\operatorname{top}}_{U,V}x_2\in \mathcal{X}_{UV}$ that satisfies
			\[
			\Phi_U(\sigma_1)\cdot^{\operatorname{top}}_{U,V}\Phi_V(\sigma_2)=\Phi_{UV}(\sigma_1\sigma_2)
			\quad\text{for all }(\sigma_1,\sigma_2)\in (U\times V)\cap \Sigma^{(2)}.
			\]
		\end{enumerate}
		The point is that the product
		formula has to depend on the choice of bisections $U$ and $V$ containing
		$\sigma_1$ and $\sigma_2$ as a result of the lack of an order-preserving
		section. That is, if $W\subseteq s(U)\cap r(V)$ is a compact open set
		containing $s(\sigma_1)=r(\sigma_2)$, then the values of
		$\Phi_{UW}(\sigma_1)$, $\Phi_{WV}(\sigma_2)$ and $\Phi_{UWV}(\sigma_1\sigma_2)$
		can differ from the values of $\Phi_U(\sigma_1)$, $\Phi_V(\sigma_2)$ and
		$\Phi_{UV}(\sigma_1\sigma_2)$. We have chosen local sections
		$\varphi_U\colon U\to\rho^{-1}(U)$ on the elements of $\Gamma_c(\Sigma)$, and the
		difference between, for example, $\Phi_U(\sigma)$ and $\Phi_{UW}(\sigma)$
		reflects the difference between the local sections $\varphi_U$ and
		$\varphi_{UW}$.
	\end{remark}

	\begin{remark}\label{RemConnection}
		Let $t_1,t_2\in T$ and let $\chi\in D_T((t_1t_2)^*t_1t_2)$. The relationship between the two types of local charts and the two local products is given by the following formula:
		\[
		\Phi_{t_1}([t_1,\beta_{t_2}(\chi)])\cdot^{\operatorname{top}}_{t_1,t_2}\Phi_{t_2}([t_2,\chi]) = \Psi_{t_1t_2}([\varphi_{t_1}(t_1)\cdot^{\operatorname{alg}}_{t_1,t_2}\varphi_{t_2}(t_2),\chi\circ (\varphi_{t_1t_2}|_{E(T)\cap(t_1t_2)\!\downarrow})^{-1}]).
		\]
	\end{remark}

	\subsection{Reformulating the topological local product}\label{subsec:reformulating}
	
	In this section we express our topological local product as a local version of the groupoid product from Proposition~\ref{prop:twistedgroupoid}. \vspace{.2truecm}
	
	To this end, let $t_1,t_2\in T$, and let $(x_1,x_2)\in \mathcal{X}_{t_1}\times_c \mathcal{X}_{t_2}$ with
	\[
	x_1=
	\left([a_{t_1},\widehat{\alpha}(\beta_{\pi(t_1t_2)}(\chi))],[\pi(t_1),\beta_{\pi(t_2)}(\chi)]\right)
	\in \mathcal{X}_{t_1}
	\]
	and
	\[
	x_2=
	\left([a_{t_2},\widehat{\alpha}(\beta_{\pi(t_2)}(\chi))],[\pi(t_2),\chi]\right)
	\in \mathcal{X}_{t_2}.
	\]
	We have
	\[
	x_1\cdot^{\operatorname{top}}_{t_1,t_2}x_2 := \left(
	\left[a_{t_1t_2},\widehat{\alpha}\!\left(\beta_{\pi(t_1t_2)}(\chi)\right)\right],
	\left[\pi(t_1t_2),\chi\right]
	\right)\in\mathcal{X}_{t_1t_2}.
	\]
	Since $\varphi:T\to N\ast_\Lambda S$ is a homomorphism and $\varphi(t_i)=(a_{t_i},\pi(t_i))$, we have
	\[
	(a_{t_1},\pi(t_1))(a_{t_2},\pi(t_2))=(a_{t_1t_2},\pi(t_1t_2)),
	\]
	and hence
	\[
	a_{t_1}\lambda_{\pi(t_1)}(a_{t_2})f(\pi(t_1),\pi(t_2))=a_{t_1t_2}.
	\]
	Thus,
	\[
	\left(
	\left[a_{t_1t_2},\widehat{\alpha}\!\left(\beta_{\pi(t_1t_2)}(\chi)\right)\right],
	\left[\pi(t_1t_2),\chi\right]
	\right)=\]
	\[\left(
	\left[a_{t_1}\lambda_{\pi(t_1)}(a_{t_2})f(\pi(t_1),\pi(t_2)),\widehat{\alpha}\!\left(\beta_{\pi(t_1)\pi(t_2)}(\chi)\right)\right],
	\left[\pi(t_1)\pi(t_2),\chi\right]
	\right).
	\]
	
	This suggests the following definition:
	
	\begin{definition}
		Let $t_1,t_2\in T$. Let $\operatorname{pr}_2\colon \mathcal{X}_{t_1}\times \mathcal{X}_{t_2}\to \mathcal{X}_{t_2}$ be the projection onto the second coordinate.
		For
		\[
		x=
		\left([a_{t_2},\widehat{\alpha}(\beta_{\pi(t_2)}(\chi))],[\pi(t_2),\chi]\right)
		\in \operatorname{pr}_2(\mathcal{X}_{t_1}\times_c \mathcal{X}_{t_2}),
		\]
		we define
		\[
		\widehat{\lambda}^{\mathrm{loc}}_{t_1,t_2}(x)
		:=
		[\lambda_{\pi(t_1)}(a_{t_2}),\widehat{\alpha}(\beta_{\pi(t_1t_2)}(\chi))]
		\in \GrpG(N).
		\]
		For $(x_1,x_2)\in \mathcal{X}_{t_1}\times_c \mathcal{X}_{t_2}$ with
		\[
		x_1=
		\left([a_{t_1},\widehat{\alpha}(\beta_{\pi(t_1t_2)}(\chi))],[\pi(t_1),\beta_{\pi(t_2)}(\chi)]\right)
		\]
		and
		\[
		x_2=
		\left([a_{t_2},\widehat{\alpha}(\beta_{\pi(t_2)}(\chi))],[\pi(t_2),\chi]\right),
		\]
		we also define
		\[
		\widehat{f}^{\mathrm{loc}}_{t_1,t_2}:\mathcal{X}_{t_1}\times_c \mathcal{X}_{t_2}\to \GrpG(N)
		\]
		by
		\[
		\widehat{f}^{\mathrm{loc}}_{t_1,t_2}(x_1,x_2)
		:=
		[f(\pi(t_1),\pi(t_2)),\widehat{\alpha}(\beta_{\pi(t_1t_2)}(\chi))].
		\]
	\end{definition}
	
	\begin{remark}\label{RemDoubt}
	We claim that for fixed $t_1,t_2\in T$ the maps $\widehat{\lambda}^{\mathrm{loc}}_{t_1,t_2}$ and $\widehat{f}^{\mathrm{loc}}_{t_1,t_2}$ are well defined. To see this, first suppose that $x'\in \operatorname{pr}_2(\mathcal{X}_{t_1}\times_c \mathcal{X}_{t_2})$ with $x=x'$. We know that
			\[
		x'=
		\left([a_{t_2},\widehat{\alpha}(\beta_{\pi(t_2)}(\chi'))],[\pi(t_2),\chi']\right)
		\]
		for some $\chi'\in D_S(\pi(t^*t))$. But then $[\pi(t_2),\chi]=[\pi(t_2),\chi']$, and so $\chi=\chi'$. Hence  $\widehat{\lambda}^{\mathrm{loc}}_{t_1,t_2}(x)=\widehat{\lambda}^{\mathrm{loc}}_{t_1,t_2}(x')$. Now suppose $(x_1',x_2')\in \mathcal{X}_{t_1}\times_c \mathcal{X}_{t_2}$ with $(x_1,x_2)=(x_1',x_2')$. We know that
		\[
		x_1'=
		\left([a_{t_1},\widehat{\alpha}(\beta_{\pi(t_1t_2)}(\chi'))],[\pi(t_1),\beta_{\pi(t_2)}(\chi')]\right)\text{ and }x_2'=
		\left([a_{t_2},\widehat{\alpha}(\beta_{\pi(t_2)}(\chi'))],[\pi(t_2),\chi']\right)
		\]
		for some $\chi'\in D_S(\pi(t_1t_2)^*\pi(t_1t_2))$. But again, we then have $[\pi(t_2),\chi]=[\pi(t_2),\chi']$, and so $\chi=\chi'$. Hence $\widehat{f}^{\mathrm{loc}}_{t_1,t_2}(x_1,x_2)=\widehat{f}^{\mathrm{loc}}_{t_1,t_2}(x_1',x_2')$.
	\end{remark}
	
	\begin{lemma}
		Let $t_1,t_2\in T$, and let $(x_1,x_2)\in \mathcal{X}_{t_1}\times_c \mathcal{X}_{t_2}$ with
		\[
		x_1=
		\left([a_{t_1},\widehat{\alpha}(\beta_{\pi(t_1t_2)}(\chi))],[\pi(t_1),\beta_{\pi(t_2)}(\chi)]\right)
		\in \mathcal{X}_{t_1}
		\]
		and
		\[
		x_2=
		\left([a_{t_2},\widehat{\alpha}(\beta_{\pi(t_2)}(\chi))],[\pi(t_2),\chi]\right)
		\in \mathcal{X}_{t_2}.
		\]
		We have
		\[
		x_1\cdot^{\mathrm{top}}_{t_1,t_2}x_2
		=
		\left(
		[a_{t_1},\widehat{\alpha}(\beta_{\pi(t_1t_2)}(\chi))]\,\widehat{\lambda}^{\mathrm{loc}}_{t_1,t_2}(x_2)\,
		\widehat{f}^{\mathrm{loc}}_{t_1,t_2}(x_1,x_2),
		\,
		[\pi(t_1t_2),\chi]
		\right).
		\]
	\end{lemma}
	
	\begin{proof}
		By definition,
		\[
		\widehat{\lambda}^{\mathrm{loc}}_{t_1,t_2}(x_2)
		=
		[\lambda_{\pi(t_1)}(a_{t_2}),\widehat{\alpha}(\beta_{\pi(t_1t_2)}(\chi))]
		\]
		and
		\[
		\widehat{f}^{\mathrm{loc}}_{t_1,t_2}(x_1,x_2)
		=
		[f(\pi(t_1),\pi(t_2)),\widehat{\alpha}(\beta_{\pi(t_1t_2)}(\chi))].
		\]
		Since all three terms have the same domain character, Lemma~\ref{lem:cliffconsequences} gives
		\[
		[a_{t_1},\widehat{\alpha}(\beta_{\pi(t_1t_2)}(\chi))]\,\widehat{\lambda}^{\mathrm{loc}}_{t_1,t_2}(x_2)\,
		\widehat{f}^{\mathrm{loc}}_{t_1,t_2}(x_1,x_2)=
		[a_{t_1}\lambda_{\pi(t_1)}(a_{t_2})f(\pi(t_1),\pi(t_2)),
		\widehat{\alpha}(\beta_{\pi(t_1t_2)}(\chi))].
		\]
		Since $\varphi:T\to N\ast_\Lambda S$ is a homomorphism and $\varphi(t_i)=(a_{t_i},\pi(t_i))$, we have
		\[
		(a_{t_1},\pi(t_1))(a_{t_2},\pi(t_2))=(a_{t_1t_2},\pi(t_1t_2)),
		\]
		and hence
		\[
		a_{t_1}\lambda_{\pi(t_1)}(a_{t_2})f(\pi(t_1),\pi(t_2))=a_{t_1t_2}.
		\]
		Therefore
		\begin{align*}
			&\left(
			[a_{t_1},\widehat{\alpha}(\beta_{\pi(t_1t_2)}(\chi))]\,\widehat{\lambda}^{\mathrm{loc}}_{t_1,t_2}(x_2)\,
			\widehat{f}^{\mathrm{loc}}_{t_1,t_2}(x_1,x_2),
			\,
			[\pi(t_1t_2),\chi]
			\right)\\
			&\hspace{8cm}=
			\left(
			[a_{t_1t_2},\widehat{\alpha}(\beta_{\pi(t_1t_2)}(\chi))],
			[\pi(t_1t_2),\chi]
			\right),
		\end{align*}
		which is exactly $x_1\cdot^{\mathrm{top}}_{t_1,t_2}x_2$.
	\end{proof}

	\subsection{A global approach}\label{subsec:globalapproach}
	  When $j$ is not order preserving we no longer have the groupoid structure on the set $\GrpG(N)\ast_{\widehat{\Lambda}}\GrpG(S)$. But the function $\widehat{\varphi}$ from Theorem~\ref{thm:mainthmwithorderpreserving} as a formula from one set to another is well defined; that is, $\widehat{\varphi}\colon \GrpG(T)\to \mathcal{X}$ given by
	 \[
	 \widehat{\varphi}([t,\chi])
	 =
	 \left(
	 [a_t,\widehat{\alpha}(\beta_{\pi(t)}(\chi\circ j|_{E(S)}))],
	 [\pi(t),\chi\circ j|_{E(S)}]
	 \right)
	 \]
	 is well defined. To see this, we need to prove that if $[t,\chi]=[u,\chi]$, then
	 \begin{enumerate}
	 	\item[(1)] $[a_t,\widehat{\alpha}(\beta_{\pi(t)}(\chi\circ j|_{E(S)}))]= [a_u,\widehat{\alpha}(\beta_{\pi(u)}(\chi\circ j|_{E(S)}))]$ , and
	 	\item[(2)] $[\pi(t),\chi\circ j|_{E(S)}]=[\pi(u),\chi\circ j|_{E(S)}]$.
	 \end{enumerate}
	 So assume $[t,\chi]=[u,\chi]$. Then there exists $r\leq t,u$ such that $\chi(r^*r)=1$. Then $\pi(r)\le \pi(t),\pi(u)$, and since $j|_{E(S)}=(\pi|_{E(T)})^{-1}$ we have
	 \[
	 \chi\circ j|_{E(S)}(\pi(r)^*\pi(r))=\chi(r^*r)=1.
	 \]
	 Hence $[\pi(t),\chi\circ j|_{E(S)}]=[\pi(u),\chi\circ j|_{E(S)}]$, and so (2) holds. Since the germs in (2) have the same range, we get
	 \[
	 \beta_{\pi(t)}(\chi\circ j|_{E(S)})=\beta_{\pi(u)}(\chi\circ j|_{E(S)})=:\eta.
	 \]
	 So to show that (1) holds we need to show that
	 \[
	 [a_t,\widehat{\alpha}(\eta)]
	 =
	 [a_u,\widehat{\alpha}(\eta)].
	 \]
	 Using the isomorphism $\varphi:T\to N*_{\Lambda}S$ given by $\varphi(x)=(a_x,\pi(x))$ we have
	 \begin{align*}
	 	r\le t,u\implies \varphi(r)\le \varphi(t),\varphi(u) &\iff (a_r,\pi(r))\le (a_t,\pi(t)),(a_u,\pi(u)) \\
	 	&\implies a_r\le a_t,a_u.
	 \end{align*}
	 We claim that $\widehat{\alpha}(\eta)(a_r^*a_r)=1$. To see this, first note that
	 \[
	 \alpha^{-1}(a_r^*a_r)=\alpha^{-1}(\iota^{-1}(j(\pi(r))r^*rj(\pi(r))^*))=\pi(j(\pi(r))r^*rj(\pi(r))^*)=\pi(rr^*).
	 \]
	 Hence
	 \begin{align*}
	 \widehat{\alpha}(\eta)(a_r^*a_r)&=\beta_{\pi(t)}(\chi\circ j|_{E(S)})(\alpha^{-1}(a_r^*a_r))\\
     &=\chi(j|_{E(S)}(\pi(t)^*\pi(rr^*)\pi(t)))\\
     &=\chi(j|_{E(S)}(\pi(t^*rr^*t))).
	 \end{align*}
	 Now note that since $r\le t$, and hence $r^*r\le t^*t$, we have
	 \[
	 t^*rr^*t = t^*(tr^*r)(tr^*r)^*t= t^*tr^*rr^*rt^*t=r^*r.
	 \]
	 So we continue the calculation above to get
	 \[
	 \widehat{\alpha}(\eta)(a_r^*a_r)=\chi(j|_{E(S)}(\pi(t^*rr^*t)))=\chi(j|_{E(S)}(\pi(r^*r)))=\chi(r^*r)=1.
	 \]
	 So the claim holds and we see that (1) holds. Hence the formula for $\widehat{\varphi}$ is well defined.

\begin{example}
    Here we trace through how our results in this section apply to the groupoid extension
    \[
	\xymatrix{\GrpH\ar[r]^{\kappa}  & \Sigma \ar[r]^{\rho} & \GrpG }
		\]
        given in Example~\ref{Exa:Example1}.  We showed that this example has no continuous global section, that is, its dual inverse semigroup extension has no order preserving section.
        Here, we have
  $\GrpG(N) \cong \GrpH$, $\GrpG(T) \cong \Sigma$ and $\GrpG(S) \cong \GrpG$ and the set \[\GrpG(N)\ast_{\widehat{\Lambda}}\GrpG(S) \cong \GrpH\ast_{\widehat{\Lambda}}\GrpG = \{(h,g) : h \in \GrpH, g\in \GrpG \text{ and } s(h)=s(g)\}. \]
  Looking at the inclusion of open bisections $B \subseteq \Sigma$ into $\GrpH\ast_{\widehat{\Lambda}}\GrpG$, we see that the topology $\tau$ on $\GrpH\ast_{\widehat{\Lambda}}\GrpG$ is the subspace topology of the product topology on $\GrpH \times \GrpG$.
Using these dual isomorphisms, the dual map $\widehat{\varphi}:\Sigma \to
\GrpH\ast_{\widehat{\Lambda}}\GrpG$ is defined so that
\[\widehat{\varphi}(\sigma) = \begin{cases} (h,\rho(\sigma)) & \text{ if }\sigma = \kappa(h)\\
((p,0), (p,1))(=((p,0),\kappa(\sigma))) & \text{otherwise}\end{cases}\] and is clearly a well-defined bijection.
\end{example}

\section{Final remark}

At the level of extensions of inverse semigroups, our examples require,
because of \cite[Proposition 3.8]{S2}, the non-existence of order-preserving
sections of the projection $\pi: T\rightarrow S$.

A simple computation shows that the existence of an order-preserving section
$j:S\rightarrow T$ of $\pi$ is equivalent to the fact that, for any $t\leq
s\in S$, $f(s, t^*t)=\alpha(t^*t)\in E(N)$; by \cite[Proposition 3.23]{DK},
this is equivalent to $f(s,e)=\alpha(ses^*)$ and $f(e,s)=\alpha(ess^*)$ for
all $s\in S$ and for all $e\in E(S)$. Thus, the non-existence of an
order-preserving section $j:S\rightarrow T$ of $\pi$ only depends on the
existence of a $2$-cocycle $f:S\times S\rightarrow N$ which does not satisfy
this requirement. Thus, a global understanding of this phenomenon requires
understanding the second cohomology group $H^2(S,N)$, as defined by Lausch
\cite{L}.

In this direction, a better approach was given in the language of ordered
groupoids (equivalent to inverse semigroups \cite[Chapter 4]{Law}) by Bainson
and Gilbert \cite{BG1}. They also showed in \cite{BG2}: (i) how to construct
every possible idempotent-bijective inverse semigroup extension of any
inverse semigroup $S$ by any \underline{abelian} Clifford semigroup $N$; (ii)
how to classify all of them via the second cohomology group $H^2(S,N)$.
\vspace{.2truecm}

Bainson and Gilbert construction turns out to be an interesting direction to
explore in its own right: using this construction, we could obtain any
inverse semigroup extension
$$
\xymatrix{N\ar[r]^{\iota}  & T \ar[r]^{\pi} & S},
$$
of an inverse semigroup $S$ by abelian Clifford semigroup $N$ \emph{up to
equivalence}, and then using Steinberg's duality \cite{S2}, obtain any
extension of any ample groupoid by any abelian group bundle
$$
\xymatrix{\GrpG(N)\ar[r]^{\widehat{\iota}}  & \GrpG(T) \ar[r]^{\widehat{\pi}} & \GrpG(S)}
$$
\emph{up to equivalence}.

\section*{AI statement}
Generative AI was not used in any way to prove or discover any of the results
in this paper.
	
	\section*{Acknowledgments}
	This work started during a visit by the third author to the University of Wollongong (Australia), the Victoria University of Wellington (New Zealand) and the University of Sydney (Australia) in Spring 2023, supported by the International Visitors Program of the Sydney Mathematical Research Institute (Australia). The third author thanks both SMRI and these universities for their support, kindness, and hospitality during the stay.


\begin{thebibliography}{9}


		\bibitem{A9}
		B. Armstrong, G.G. de Castro, L.O. Clark, K. Courtney, Y.-F. Lin, K. McCormick, J. Ramagge, A. Sims, B. Steinberg \textit{Reconstruction of Twisted Steinberg Algebras}, Int. Math. Res. Not. IMRN (2023), no. 3, 2474--2542.


		\bibitem{A6}
		B. Armstrong, L.O. Clark, K. Courtney, Y.-F. Lin, K. McCormick, J. Ramagge, \textit{Twisted Steinberg algebras}, J. Pure Appl. Algebra \textbf{226} (2022), no. 3, Paper No. 106853, 33 pp.
		
        \bibitem{ACaHJL}B. Armstrong, L.O. Clark, A. an Huef, M. Jones, Y.-F. Lin, \emph{Filtering germs: groupoids associated to inverse semigroups}, Expo. Math. {\bf 40} (2022), no.~1, 127--139.
		

		\bibitem{BG1}
		B.O. Bainson, N.D. Gilbert, \textit{Cohomology and extensions of ordered groupoids}, J. Alg. Appl. \textbf{21} (2022), No. 9, Paper No. 2250177, 23 pp.
		
		\bibitem{BG2}
		B.O. Bainson, N.D. Gilbert, \textit{Cohomology and extensions of ordered groupoids}, The twenty-sixth NBSAN meeting, St Andrews on 13th July 2017, slides of N.D. Gilbert presentation, https://personalpages.manchester.ac.uk/staff/Mark.Kambites/events/nbsan/nbsan26\_gilbert.pdf.
		
		\bibitem{DK}
		M. Dokuchaev, M. Khrypchenko, \textit{Twisted partial actions and extensions of semilattices of groups by groups}, Inter. J. Algebra Comput. \textbf{27} (2017), no. 7, 887--933.
		
		\bibitem{KKLRU}
		M. Kennedy, S.-J. Kim, X. Li, S. Raum, D. Ursu, \textit{The ideal intersection property for essential groupoid $C^*$-algebras}, arXiv:2107.03980v2 (2021).

        \bibitem{Kumjian}
		A. Kumjian, \textit{On $C^*$-diagonals}, Can. J. Math. \textbf{38} (1986), 969--1008.
		
		
		\bibitem{L}
		H. Lausch, \textit{Cohomology of inverse semigroups}, J. Algebra \textbf{35} (1975), 273--303.
		
		\bibitem{Law}
		M.V. Lawson, ``Inverse semigroups. The Theory of Partial Symmetries'', World Scientific Publishing Co., Inc., River Edge, NJ, 1998, xiv+411 pp.
		
		

        \bibitem{Pass1}
		D. Passman, ``The algebraic structure of group rings", Krieger Publishing Co., Inc., Melbourne, FL, 1985, xiv+734 pp.

		\bibitem{Pass2}
		D. Passman, ``Infinite crossed products", Pure Appl. Math., 135, Academic Press, Inc., Boston, MA, 1989, xii+468 pp.
		
		\bibitem{Renault}
		J. Renault, ``A Groupoid Approach to C*-Algebras", Lecture Notes in Math., vol. 793, Springer-Verlag, New York, 1980.
		
		\bibitem{SW16}
		A. Sims and D.P. Williams, \textit{The primitive ideals of some \'etale groupoid $C^*$-algebras}, Algebr. Represent. Theory \textbf{19} (2016), 255--276.
		
		\bibitem{S1}
		B. Steinberg, \textit{A groupoid approach to discrete inverse semigroup algebras}, Adv. Math. \textbf{223} (2010), 689--725.
		
		\bibitem{S2}
		B. Steinberg, \textit{Twists, crossed products and inverse semigroup cohomology}, J. Aust. Math. Soc. \textbf{114} (2023), no. 2, 253--288.
		
		\bibitem{SZ}
		B. Steinberg, N. Szak\'acs, \textit{Simplicity of inverse semigroup and \'etale groupoid algebras}, Adv. Math. \textbf{380} (2021), Paper No. 107611, 55 pp.
		
	\end{thebibliography}
\end{document}